\documentclass[11pt, a4paper]{article}        
\usepackage[left=2.5cm,right=3cm,top=2.5cm,bottom=2cm,includeheadfoot]{geometry}

\usepackage[english]{babel}
\usepackage[utf8]{inputenc}

\usepackage{amsmath}
\usepackage{amsfonts}
\usepackage{amssymb}
\usepackage{amsthm} 
\usepackage{stmaryrd}
\usepackage{dsfont}

\usepackage[bf,hang,nooneline,font=footnotesize]{caption}						
\usepackage{graphicx}														
\usepackage{tabularx}															
\usepackage{multirow}															
\usepackage{booktabs}
\usepackage[labelformat=brace, font=scriptsize, textfont=it]{subcaption}

\theoremstyle{plain}
\newtheorem{theorem}{Theorem}[section]

\newtheorem{corollary}[theorem]{Corollary}
\newtheorem{prop}[theorem]{Proposition}
\newtheorem{assumption}[theorem]{Assumption}

\theoremstyle{definition}
\newtheorem{definition}[theorem]{Definition}

\theoremstyle{remark}

\newtheorem{remark}[theorem]{Remark}

\newtheorem{example}[theorem]{Example}

\newcommand\xqed[1]{%
	\leavevmode\unskip\penalty9999 \hbox{}\nobreak\hfill
	\quad\hbox{#1}}
\newcommand\defEnd[0]{\xqed{$\diamond$}}

\usepackage{url}

\usepackage{hyperref}

\usepackage{graphicx}
\usepackage{xcolor}

\usepackage{array}
\usepackage[titletoc,title]{appendix}

\usepackage{titlesec}
\titleformat{\section}[runin]{\normalfont\bfseries}{\thesection .}{0.5em}{}
\titleformat{\subsection}[runin]{\normalfont\bfseries}{\thesubsection .}{0.5em}{}

\usepackage{changes}

\usepackage{todonotes}

\DeclareMathOperator*{\argmax}{argmax}
\DeclareMathOperator*{\med}{med}

\begin{document}

	\title{
		Markovian randomized equilibria for general Markovian Dynkin games in discrete time 
	}
	
	\author{S\"oren Christensen\thanks{Kiel University, Department of Mathematics, \emph{christensen@math.uni-kiel.de.}}\hspace{1cm} Kristoffer Lindensjö\thanks{Stockholm University, Department of Mathematics, \emph{kristoffer.lindensjo@math.su.se}}\;\; \\
		Berenice Anne Neumann\thanks{Trier University, Department IV - Mathematics, \emph{neumannb@uni-trier.de}}}
	\maketitle

	\allowdisplaybreaks

	\begin{abstract}
		We study a general formulation of the classical two-player Dynkin game in a discrete time Markovian setting. 
        We identify an appropriate class of mixed strategies -- \textit{Markovian randomized stopping times} -- in which players stop at any given state with a state-dependent probability.
        One main result is an explicit characterization of Wald-Bellman-type for Nash equilibria based on this notion of randomization. 
        In particular, we derive a novel characterization of randomized equilibria in zero-sum Dynkin games, which we use to (i) establish the existence and explicit construction of Markovian randomized equilibria, (ii) provide necessary and sufficient conditions for the non-existence of pure strategy equilibria, and (iii) construct an example that admits a unique randomized equilibrium but no pure one.
		We also provide existence and characterization results in the symmetric version of our game. 
		Finally, we establish existence of a characterizable equilibrium in Markovian randomized stopping times for the general game formulation under the assumption that the state space is countable. \\

		

	\end{abstract}

\section{Introduction}\label{sec:intro}
Starting with the seminal paper \cite{DynkinStoppingGame}, Dynkin games have been studied extensively. 
In the present paper we restrict our attention to Dynkin games for Markov processes in discrete time. Dynkin games in continuous time have been studied in, 
e.g., \cite{attard2018nonzero,BensoussanFriedman1977Nonzero,deanglis2018nash,hamadene2010continuous,laraki2005value,riedel2017subgame}, and for an overview we refer the reader to \cite{KiferSurvey}. 
	
In the general discrete time Dynkin game formulation the rewards of the two players $i=1,2$ read, for a given stopping time pair $(\tau_1, \tau_2)$, as 
\begin{align}\label{general-rewards}
\mathbb{E} \left[ F_{\tau_i}^i \mathbb{I}_{\{\tau_i < \tau_j\}} + G_{\tau_j}^i \mathbb{I}_{\{\tau_j < \tau_i\}} + H_{\tau_i}^i \mathbb{I}_{\{\tau_i = \tau_j\}} \right]
\end{align}
for $i,j =1,2$ with $j \neq i$, where $F^i, G^i,H^i, i=1,2,$ are integrable discrete time processes (with a suitable interpretation of $H_n^i$ for $n=\infty$). 
Here, 
	$F^i$ corresponds to the payoff for player $i$ when player $i$ stops first, 
	$G^i$ corresponds to the payoff for player $i$ when the other player $j \neq i$ stops first, and 
	$H^i$ corresponds to the payoff for player $i$ when both players stop at the same time. 
	
A pair of stopping times is for this game said to be a Nash equilibrium if it satisfies the usual assumption of sub-optimal deviation for each of the two players; 
in particular, $(\tau_1,\tau_2)$ is a Nash equilibrium if 
$\tau_i$ maximizes \eqref{general-rewards} over admissible stopping times assuming that $\tau_j, j \neq i$ is fixed, for $i=1,2$. 
Similarly, we say, for a fixed $\epsilon>0$, that we have an $\epsilon$-equilibrium if there exists a pair $(\tau_1,\tau_2)$ such that $\tau_i$ achieves the supremum up to $\epsilon$ in \eqref{general-rewards}, assuming that $\tau_j, i \neq j$ is fixed, for $i=1,2$.

The most studied class of Dynkin games are zero-sum games, i.e., games where one player's gain is the other player's loss. 
Mathematically, this means $F^1 = - G^2$, $H^1 =  - H^2$ and $G^1 = - F^2$. If a game does not satisfy this we call it a non-zero-sum game.
	
Dynkin games are well-understood whenever the payoffs are ordered such that $F^1 \le H^1 \le G^1$ and $F^2 \le H^2 \le G^2$, both in the zero-sum as well as in the non-zero-sum formulation:  
in this setting it is  sufficient to consider pure strategies (i.e., stopping without randomization). 
For the zero-sum version Neveu \cite{NeveuMartingales} proved  existence of $\epsilon$-equilibria, for every $\epsilon>0$, in pure strategies and provided a characterization. 
Thereafter, existence and characterization of a Nash equilibrium for the zero-sum game (again in pure strategies) has been established by Ohtsubo in \cite{OhtsuboTerminating}. 
In the non-zero-sum case Ohtsubo \cite{OhtsuboNonzeroSum} provides verification results and constructs explicit (pure-strategy) Nash equilibria for finite time horizons and establishes the existence of a Nash equilibrium for a Markovian game under the condition that $G^1$ is a supermartingale. In \cite{OhtsuboNonZeroMonotone} Ohtsubo explicitly constructs (pure-strategy) Nash equilibria for a monotone problem.   
Also non-constructive existence results for pure-strategy Nash equilibria have been established. Namely,  in \cite{MorimotoNonZero} the existence of Nash equilibria for non-zero-sum games under the assumption that $F^i$ is a submartingale and $G^i$ is a supermartingale is established and in \cite{HamadeneDiscrete} the existence of 
Nash equilibria for non-zero-sum $n$-player games is proved.

However, if we drop the assumption that $F^1 \le H^1 \le G^1$ and $F^2 \le H^2 \le G^2$, the situation becomes more involved. First of all, we now need to consider a much larger and more difficult class of strategies, namely randomized stopping strategies \cite{YasudaRandomized}. However, even for this larger strategy class it can be the case that no Nash equilibrium exists even in the case where an $\epsilon$-equilibrium, for every $\epsilon>0$, does exist, see \cite{Shmaya_Deterministic_Epsilon_Existence}.  
Nonetheless, the existence of $\epsilon$-equilibria, for every $\epsilon>0$, for general games can be established requiring only  integrability conditions. 
Indeed, in \cite{YasudaRandomized} the existence of a value for zero-sum games 
with finite time horizon or discounting is established. Here, as usual for zero-sum games, existence of a value means that the expected reward when first taking the supremum over $\tau_1$ and second taking the infimum $\tau_2$ is the same 
as when first considering infimum over $\tau_2$ and second the supremum over $\tau_1$.
Moreover, the existence of $\epsilon$-equilibria, for every $\epsilon>0$, has been established in \cite{RosenbergSolanVieilleZeroSum} for zero-sum games and in \cite{ShmayaSolanNonZeroExistence} for  non-zero-sum games. However, these results are non-constructive. Indeed, there are, according to the knowledge of the authors, no general results on the characterization of randomized equilibria of Dynkin games.

In developing the theory for ordinary stopping problems, one typically assumes an additional Markov structure, see, e.g., \cite{peskir2006optimal, shiryaev2007optimal}. This provides a general framework for solving problems concretely, since one can restrict the class of stopping times to state-dependent first entrance times.  Surprisingly, this view has never previously been systematically adopted in the treatment of Dynkin games. For discrete time Markovian Dynkin games (see Section~\ref{sec:Markovian-Dynkin} for the general formulation studied in the present paper), only specific formulations have, according to the knowledge of the authors, been considered: 
\cite{FridMarkovChain} considers a zero-sum game where depending on the state of the process only one player can stop at every time step, 
\cite{ElbakidzeMarkovGame} considers the zero-sum version with the additional assumption that $F^1 \le H^1 \le G^1$, 
\cite{DomanskyMarkovSpecialCase2, DomanskyMarkovSpecialCase} both consider a zero-sum game for a particular Markov chain, and 
\cite{FerensteinMarkovGame} considers a non-zero-sum $n$-player game for a particular class of transition rates and allows only Markovian strategies. 
In \textit{continuous time}, Markovian zero-sum games have been considered under the condition $F^1 \le H^1 \le G^1$. Here, too, it suffices to consider pure strategies, and in this case there are general existence results 
\cite{Ekstrom_Peskir_Markov}, as well as characterization results in terms of (quasi)-variational inequalities for diffusions \cite{Bensoussan_Friedman_1974,Friedman_1973} and super- and subharmonic functions \cite{Peskir_Semiharmonic}. Moreover, in \cite{deanglis2018nash} the existence of equilibria in (pure) threshold strategies has been established for continuous time, non-zero-sum games with underlying diffusion whenever $F^i \le H^i \le G^i$ as well as additional conditions hold.
In \cite{christensen2024general}, a complete equilibrium solution to a general formulation, i.e., without ordering conditions for the payoffs, of the zero-sum game with an underlying one-dimensional diffusion is presented.

\subsection{The Markovian Dynkin game and contributions}\label{sec:Markovian-Dynkin}
Let us introduce the general discrete time Markovian Dynkin games that we investigate in this paper: 
let $X=(X_n)_{n \in \mathbb{N}_0}$ be a homogeneous Markov process with a Markov kernel $\Pi$ on a probability space $\left(\Omega, {\cal F}, \mathbb{P}_{x}\right)$ with state space $E$. The associated expectations under $X_0=x \in E$ are denoted by $\mathbb{E}_{x}$. 
	In this setting we consider a general Dynkin game with two players $i=1,2,$ each choosing a stopping time $\tau_i$.  
	In particular, for a given stopping time pair $(\tau_1,\tau_2)$ the corresponding expected rewards of the players are 
	\begin{align*}
		J_i({x}; \tau_1,\tau_2): = \mathbb{E}_{x} 
		\left[ \alpha^{\tau_i} f_i(X_{\tau_i}) \mathbb{I}_{\{\tau_i < \tau_j\}} + \alpha^{\tau_j} g_i(X_{\tau_j}) \mathbb{I}_{\{\tau_j < \tau_i\}} + \alpha^{\tau_i} h_i(X_{\tau_i}) 
		\mathbb{I}_{\{\tau_i= \tau_j < \infty\}} \right]
	\end{align*}
	for $i,j=1,2$ with $j \neq i$, where $\alpha$ is a constant (discount factor)  satisfying $0 < \alpha < 1$ and $f_i,g_i,h_i: E \rightarrow \mathbb{R},i=1,2,$ are measurable functions; see Section~\ref{sec:model} for details regarding, e.g., the set of admissible stopping times and integrability assumptions, as well as the Nash equilibrium definition.

	The main contributions of the present paper are: 
	
	\begin{enumerate}
		
		\item Our investigations indicate that a general theory for discrete time Dynkin games for underlying Markov processes should be built on a very natural class of randomized stopping times -- which we refer to as \textit{Markovian randomized stopping times}. The interpretation of a Markovian randomized stopping time is that it corresponds to stopping at each date according to a probability that depends only on the value of the state process at that date; see Section~\ref{sec:model} for details. 
		Indeed, our investigations strongly indicate that Markovian randomized stopping times correspond  to the right type of randomized (also known as \textit{mixed}) strategy for discrete time stopping games, not only from an intuitive, but also from a mathematical standpoint. 
		In particular, by considering Markovian randomized stopping times we are able to explicitly characterize and construct equilibria, 
		as well as to prove equilibrium existence results (see the items below). We remark that we do not restrict the \textit{admissible} stopping times to be of Markovian randomized type, but instead allow a general class of stopping times; see Section~\ref{sec:model}.

		\item For the general game formulation we provide an equilibrium characterization, as well as a verification, result formulated in terms of a system of Wald-Bellman type equations, 
		for both randomized and pure equilibria. The main novelty is that this provides an \textit{explicit} equilibrium characterization encompassing also randomized equilibria. 
		
		\item  Relying on our equilibrium characterization we study two specifications of our game. 
		For the well-known zero-sum game, we  
		(i) establish the existence of a global Markovian equilibrium without any ordering condition and provide a corresponding explicit equilibrium construction,   
		(ii) obtain necessary and sufficient conditions for the non-existence of pure equilibria, and
		(iii) construct an explicit example with a unique randomized, but no pure equilibrium. 
		For the symmetric specification of our game -- meaning that the two players have identical payoff functions -- we establish that a (possibly randomized) \textit{symmetric} equilibrium exists under a certain condition on the payoff functions, and we provide a corresponding equilibrium construction in terms of an associated optimal stopping problem.

		\item In addition to the existence results for the zero-sum and symmetric games, we establish existence of an equilibrium in Markovian randomized stopping times for the general game formulation under the assumption that the state space $E$ is countable. 
		This stands in contrast to the known existence result in the literature in two ways. First of all, for games without discounting only existence of $\epsilon$-equilibria has been established and there are counter-examples showing that Nash equilibria do not exist even for deterministic and stationary rewards \cite{Shmaya_Deterministic_Epsilon_Existence}. We now establish the existence of a Nash equilibrium under the assumption of discounting in connection with a slightly stronger integrability condition (see Remark~\ref{rem:integrability}). Secondly, in our formulation the equilibrium can be characterized (see item 2. above).

	\end{enumerate}

	In Section~\ref{sec:model}, we specify the mathematical model and the game formulation. Further related literature is reviewed in Remarks~\ref{rem:literature-mixed-stopping} and \ref{rem:literature-mixed-stopping-cont}. 
	In Section~\ref{sec:Optimization} we study the best response mapping for our game. 
	Section~\ref{sec:ver-thm} provides the equilibrium characterization.  
	The zero-sum game is studied in Section~\ref{sec:zero-sum}. 
	The symmetric game is studied in Sections~\ref{sec:symmetric} and \ref{sec:countable-symNE}.  
	General equilibrium existence for countable state spaces is established in Section~\ref{sec:Existence}.

	\section{Randomized Markovian Stopping Times, Nash Equilibrium and Assumptions}\label{sec:model}
	
	Let us first define what we mean by Markovian randomized stopping times.
	
	\begin{definition}[Markovian randomized stopping times] 
		Let $(\xi_n^{(i)})_{n \in \mathbb{N}_0},i=1,2,$ be sequences of iid random variables with $\xi_n^{(i)} \sim U(0,1)$ which are independent of the state process $X$
		and supported by our probability space. A Markovian  randomized stopping time (for player $i=1,2$) is given by 
		\begin{align}\label{eq:M-rand-stop-time}
			\tau^{p^{(i)}} = \inf \{ n \in \mathbb{N}_0: p^{(i)}(X_n) \ge \xi^{(i)}_n\}
		\end{align} 
		where $p^{(i)}:E \rightarrow [0,1]$ is a deterministic measurable function. We use the standard convention that $\inf \emptyset = \infty$. \defEnd
	\end{definition}

	We identify a Markovian randomized stopping time \eqref{eq:M-rand-stop-time} with the associated function $p^{(i)}$ and will often, 
	for example, refer to $\left(p^{(1)},p^{(2)}\right)$ as a pair of stopping strategies.

	\begin{remark}[Markovian randomized stopping] \label{rem:literature-mixed-stopping}
		The interpretation of  a Markovian randomized stopping time is that 
		$\xi_n^{(i)}$ is a randomization device (corresponding to a coin flip) that player $i$ employs to randomize the stopping decision at time $n$;  
		where the choice of function $x \mapsto p^{(i)}(x)$ determines the probability (the bias of the coin) of stopping when visiting each individual state $x \in E$. 
		Hence, a Markovian randomized stopping time  corresponds essentially to randomizing by stopping at each date independently according to a probability that depends only on the value of the state process at that date. 
		We remark that this type of stopping time has previously been studied in the context of an  $n$-player stopping game in \cite{FerensteinMarkovGame}, 
		as well as in the context of \emph{time-inconsistent stopping problems}; 
		see 
		\cite[Example 2.9]{christensen2018finding}, 
		\cite{bayraktar2019time}, 
		and \cite{christensen2020Timemyopic}. Markovian randomized stopping in continuous time is discussed in Remark~\ref{rem:literature-mixed-stopping-cont} (below).
		\defEnd
	\end{remark}
	
	\begin{definition}[Admissible stopping times]\label{def:Admissible}
		For each player $i=1,2$, the set of admissible stopping times, denoted by ${\cal T}_i$, is defined as stopping times with respect to the filtration 
		$\sigma\left(X_0,...,X_n,\xi^{(i)}_{0},...,\xi^{(i)}_{n}\right),n \in \mathbb{N}_0$. 
		\defEnd
	\end{definition}
	\begin{remark}   
		The interpretation of an admissible stopping time is that the decision to stop for each player 
		can be based on the current and previous values of the state process and the randomization device of that player. 
		It is clear that Markovian randomized stopping times are admissible. \defEnd
	\end{remark}
	Both players in our game are assumed (without loss of generality) to be maximizers and we define our notion of equilibrium accordingly.
	\begin{definition}[Nash equilibrium] \label{def:NE}
		A pair of admissible stopping times $(\tilde \tau_1, \tilde \tau_2)$, $\tilde \tau_i\in {\cal T}_i ,i=1,2$ is a  \emph{Nash equilibrium for ${x} \in E$} if
		\begin{align*}
			J_1({x};\tilde \tau_1, \tilde \tau_2) = \sup_{\tau_1\in {\cal T}_1} J_1({x};\tau_1, \tilde \tau_2),	\\
			J_2({x};\tilde \tau_1, \tilde \tau_2) = \sup_{\tau_2\in {\cal T}_2} J_2({x};\tilde \tau_1, \tau_2). 
		\end{align*}
		If a pair of stopping strategies $\left(p^{(1)},p^{(2)}\right)$ corresponds to a Nash equilibrium for $x$ then we refer to it as a 
		\emph{Markovian randomized equilibrium for $x$}. 
		If $\left(p^{(1)},p^{(2)}\right)$ corresponds to a Nash equilibrium \emph{for all} $x \in E$ then we refer to it as a \emph{global Markovian randomized equilibrium}. 
		\defEnd
	\end{definition}  
	
	The aim of the present paper is to study global Markovian randomized equilibria $\left(p^{(1)},p^{(2)}\right)$. 
	We remark that we will often let the \emph{randomized} be implicit when referring to, e.g., Markovian randomized stopping times. 
	\begin{remark}
		From the perspective of classical game theory, it holds that a global Markovian randomized equilibrium $(p^{(1)},p^{(2)})$ is a \textit{subgame perfect equilibrium} as well as a \textit{Markov perfect equilibrium}; for a reference for these terms see e.g., \cite{FudenbergGames1991}. Indeed, subgames for discrete time stopping games are given as those games that start at time $n$ given the realizations of $X_0, X_1, \ldots, X_n$ for any $n \in \mathbb{N}$. Moreover, since the expected rewards for a global Markovian randomized equilibrium depend only on the current state and $(p^{(1)},p^{(2)})$, which by definition is an equilibrium for any $x \in E$, it is clear that $(p^{(1)},p^{(2)})$ is subgame perfect. Noting also that the strategies $(p^{(1)},p^{(2)})$ are Markov in the sense of game theory, i.e., they depend only on past events that are payoff relevant, we immediately see that the equilibrium is also Markov perfect.
		We refer to our equilibria as Markovian randomized equilibria  to emphasize that not only are they Markov perfect, but that the randomization itself has a particular form, namely that the probability of stopping depends only on the current state.\defEnd
	\end{remark}

	If a Markovian stopping strategy $p^{(i)}$ is such that $p^{(i)}(x) \in \{0,1\}$ for each $x \in E$ then there is effectively no randomization and we hence refer to it as \emph{pure}. 
	A global Markovian equilibrium $\left(p^{(1)},p^{(2)}\right)$ is classified as pure if it is comprised of pure Markovian stopping strategies.  
	Note that pure Markovian stopping strategies correspond to entry times of $X$ into subsets in the state space $E$. 
	
	Our approach to studying the present game relies on the consideration of a process $\tilde X=(\tilde X_n)_{n \in \mathbb{N}_0}$ defined (path-wise) as the process $X$ killed at each $n$  with probability $1-\alpha$; which we assume is supported by our probability space. As usual, if $\tilde{X}$ is killed, it means that it is sent to a cemetery state $K$, where it is absorbed and where all functions take the value $0$. 
	In particular, it holds that $X_0= \tilde X_0$ and $X_n= \tilde X_n$ for each $n > 0$ with $\tilde X_n \neq K$. Throughout the paper, we assume that the following integrability conditions, well known from optimal stopping theory, are satisfied.
	
	\begin{assumption}\label{assum:bounded-functions} For the killed process $\tilde X$ it holds, for $i=1,2$, that
		\begin{equation}
			\label{eq:Assumption_Boundedness}
			\sup_{n \in \mathbb{N}_0}|f_i(\tilde{X}_n)|,\,\sup_{n \in \mathbb{N}_0} |g_i(\tilde{X}_n)|,\,\sup_{n \in \mathbb{N}_0} |h_i(\tilde{X}_n)|  \in L^1.
		\end{equation} \defEnd
	\end{assumption}
	Note that this assumption implies that
	\begin{equation}
		\label{eq:weak_integrability_condition}
		\sup_{n \in \mathbb{N}_0} \alpha^n|f_i(X_n)| ,\, \sup_{n \in \mathbb{N}_0} \alpha^n|g_i(X_n)| ,\,\sup_{n \in \mathbb{N}_0} \alpha^n|h_i(X_n)| \in L^1,
	\end{equation}
	which we remark is the standard condition for uniform integrability of the payoff processes for the associated ordinary discounted stopping problems. For later use we introduce the notation 
	\begin{align*}
		M = \sup_{n \in \mathbb{N}_0, i \in \{1,2\}}  \max \{ |f_i(\tilde{X}_n)|, |g_i(\tilde{X}_n)|, |h_i(\tilde{X}_n)| \}
	\end{align*} and note that Assumption~\ref{assum:bounded-functions} implies that $M \in L^1$. 
	
\begin{remark} \label{rem:integrability}
We note that the integrability condition in Assumption~\ref{assum:bounded-functions} is only slightly stronger than \eqref{eq:weak_integrability_condition}. 
Namely, it holds that if $\mathbb{E}_x [\sup_{n \in \mathbb{N}_0} \beta ^n f(X_n)]< \infty$ for some $\beta \in (\alpha,1]$, then 
$\mathbb{E}_x[\sup_{n \in \mathbb{N}_0} f(\tilde{X}_n)] < \infty$ for all non-negative functions $f:E \rightarrow \mathbb R$. Indeed, let $T$ be a geometric random variable with parameter $\alpha$ independent of $(X_n)_{n \in \mathbb{N}}$ and define $\gamma = \alpha/\beta$. Then
		\begin{align*}
			\mathbb{E}_x \Big[ \sup_{n \in \mathbb{N}_0} f(\tilde{X}_n) \Big] 
			&= \mathbb{E}_x \Big[ \sup_{n \le T} f(X_n) \Big] = \mathbb{E}_x \Big[ \sum_{k=0}^\infty (1-\alpha)\alpha^k \sup_{n \le k} f(X_n) \Big]\\ 
			&\leq \mathbb{E}_x \Big[ \sum_{k=0}^\infty \gamma^k \sup_{n \le k} \beta^n f(X_n) \Big] \le \mathbb{E}_x \Big[ \sum_{k=0}^\infty \gamma^k \sup_{n \in \mathbb{N}_0} \beta^n f(X_n) \Big]\\
			& = \frac{1}{1-\gamma} \mathbb{E}_x \Big[ \sup_{n \in \mathbb{N}_0} \beta^n f(X_n) \Big] < \infty.
		\end{align*}
	In Example~\ref{ex:no-NE} (below) we argue how our integrability condition is essential for obtaining general equilibrium existence. 
	\defEnd
	\end{remark}

	\begin{remark}[On Markovian randomized stopping in continuous time]\label{rem:literature-mixed-stopping-cont}
		A continuous time interpretation of Markovian randomized stopping is to stop according to a state dependent stopping rate; see \cite[Section 2.1]{christensen2020time} for a motivation. 
		Typically this state-dependent intensity is of Lebesgue density type (see \cite{christensen2020time} for a definition). 
		Generally there is however no need to restrict attention to such intensities; in fact, it turns out that allowing the intensity to increase in a singular fashion using a local time construction may facilitate the existence of equilibria (in \cite {bodnariu2024local} a general construction of this type of randomized stopping is introduced in the context of studying a time-inconsistent stopping problem in an SDE setting). 
		In \cite[Example 5.4]{ekstrom2017dynkin}, a similar local-time construction is used to find randomized equilibria for a two-player stopping game in a setting where the players essentially agree that the state process is driven by an SDE, but disagree on what the drift is. 
		In \cite{decamps2022mixed} this type of local-time-based randomized stopping time is used in the study of a general formulation of a two-player stopping game corresponding to the \emph{war of attrition} in an SDE setting.  
In \cite{christensen2024general} similar continuous time Markovian randomized stopping times are used to arrive at a complete equilibrium solution to a general formulation, i.e., without ordering conditions for the payoffs, of the zero-sum game for a one-dimensional diffusion.
It seems to us an interesting question for future research which class of randomized stopping times one should use for a general theory of Markovian Dynkin games in continuous time.
		\defEnd
	\end{remark}

	\section{Best Response Mapping}\label{sec:Optimization}
	In this section we define and establish properties of the one-player \emph{best response mapping} corresponding to our game. 
	The analysis is without loss of generality carried out from the view-point of player $1$, 
	i.e., the exact same analysis can also be carried out from the view-point of player $2$. 
	The main results are 
	Propositions~\ref{thm:CharacterizationValues}  and \ref{thm:characterization_optimal_strategies} which are essential in subsequent sections.

	In order to define this best response mapping we consider the optimal (in the usual sense) stopping problem that player $1$ faces when player $2$ employs a fixed Markovian stopping strategy $p^{(2)}:E \rightarrow [0,1]$, i.e.,
	\begin{align} \label{sec3:optstop1}
		\sup_{\tau_1\in {\cal T}_1} J_1\left({x};\tau_1, \tau^{p^{(2)}}\right), {x} \in E.
	\end{align}
	It turns out that a solution to this problem can be found in the class of Markovian stopping times; see Proposition~\ref{thm:characterization_optimal_strategies} below.  
	Hence, we may define a best-response mapping corresponding to problem \eqref{sec3:optstop1} for each \emph{fixed} $x \in E$ restricted to Markovian stopping times, i.e., as the (point-to-set) mapping defined according to 
	\begin{align*} 
		\text{BR}^{(1)}: E \times \mathcal{M}(E,[0,1])  &\rightarrow {\cal P}\left(\mathcal{M}(E,[0,1]) \right) \\
			\left({x},p^{(2)}\right)  &\mapsto \text{BR}^{(1)}\left({x}, p^{(2)}\right):= \argmax_{p^{(1)}} J_1\left({x};\tau^{p^{(1)}}, \tau^{p^{(2)}}\right)
	\end{align*}
	where $\mathcal{M}(E,[0,1])$ denotes the set of measurable functions $E \rightarrow [0,1]$ and 
	${\cal P}\left(\mathcal{M}(E,[0,1]) \right)$ denotes the corresponding power set.
	The interpretation is that $\text{BR}^{(1)}\left(x,p^{(2)}\right)$ is the set of optimal stopping strategies of Markovian type for problem 
	\eqref{sec3:optstop1}, i.e., the set of optimal stopping strategy functions $E \rightarrow [0,1]$ for player $1$, given that player $2$ employs the strategy $p^{(2)}$, for a  fixed initial state $x$.
	
	However, our main interest in the present paper is to study Markovian strategies that are equilibria \emph{for all} initial states (i.e., global Markovian equilibria, cf. Definition~\ref{def:NE}), 
	and for this reason the following definition  
	-- which gives as output the set of Markovian strategies which are optimal responses to a fixed strategy $p^{(2)}$ \emph{for all} $x \in E$ -- 
	will, as we shall see, be the right one. 
	\begin{definition}[One-player best response mapping]\label{def:best-response-mapping} We call the set-valued mapping 		
		\begin{align}
			\label{best-response-mapping}
			\text{BR}^{(1)}: \mathcal{M}(E,[0,1]) &\rightarrow {\cal P}\left(\mathcal{M}(E,[0,1]) \right)\\
				p^{(2)} & \mapsto \text{BR}^{(1)}\left(p^{(2)}\right):=\bigcap_{{x} \in E} \text{BR}^{(1)}\left({x},p^{(2)}\right),\nonumber
		\end{align}
	the one-player best response mapping. \defEnd
	\end{definition}
	The following result will be proved at the end of this section.
	\begin{prop}
		\label{thm:CharacterizationValues}
		For any fixed stopping strategy $p^{(2)}:E \rightarrow [0,1]$, the set of best responses 
		$\text{BR}^{(1)}\left(p^{(2)}\right)$  is non-empty, convex, and closed (in the product topology). 
	\end{prop}
	In order for \eqref{best-response-mapping} to be a suitable definition we need, as we have mentioned, that a maximizer in \eqref{sec3:optstop1} can be attained in the set of Markovian stopping strategies (this result is contained in Proposition~\ref{thm:characterization_optimal_strategies} below). 
	Our approach to showing that this is the case is to establish a link between problem \eqref{sec3:optstop1} and an associated optimal stopping problem (see \eqref{sec3:optstop2}, below), which we can study using standard methods. 
	
Note that \eqref{sec3:optstop1} is a non-standard optimal stopping problem in the sense that the underlying process is exogenously stopped randomly according to the strategy $p^{(2)}$. A main feature of the associated problem \eqref{sec3:optstop2} is that it will be constructed without discounting which in particular means that the non-standard stopping feature corresponding to $p^{(2)}$ can be treated as a more standard \emph{absorption} feature for an associated state process $\hat X$ (see the below for details).

	To be able to construct the associated stopping problem we define a new reward function -- see \eqref{r-reward} below -- and an associated Markov process 
	$\hat{X}=(\hat{X}_n)_{n \in \mathbb{N}_0}$ defined on a state space
	\[
	\hat{E} = (E \times \{C,S\}) \cup \{K\}
	\] 
	in a \emph{path-wise} manner based on 
	our state process $X$,
	its killed counterpart $\tilde X$, and 
	the stopping behavior of player $2$. We denote elements in $\hat E$ by $\hat x$. Before giving the formal definition of $\hat X$ let us give an interpretation of the state space: 
	(i) if $\hat{X}_n=(x,S)$ then player $2$ has stopped before time $n$ (i.e., $\tau^{p^{(2)}}<n$) and $x=X_{\tau^{p^{(2)}}}$, where $(x,S),x \in E,$ are absorbing states, 
	(ii) if $\hat{X}_n=K$, then killing has occurred (recall that $K$ is a cemetery state), and  
	(iii) if $\hat{X}_n=(x,C)$ then neither killing nor stopping has occurred. 
	The formal definition of $\hat{X}$ is as follows:

	\begin{definition}[Killed and absorbed version of the state process]  \label{def:X-hat}
		The initial value of the process $\hat{X}=(\hat{X}_n)_{n \in \mathbb{N}_0}$ is $\hat{X}_0 = (X_0,C)$.  For each $n>1$:
		\begin{itemize}
			\item  if $\hat X_{n-1} = (X_{n-1},C)$ \vspace{-2mm}
			\begin{itemize}
				\item and $\tau^{p^{(2)}}=n-1$, then $\hat X_n = (X_{n-1},S)$  \vspace{-2mm}
				\item  and $\tau^{p^{(2)}} > n-1$ and $\tilde X_n = K$, then $\hat X_n = K$ (recall that $\tilde X$ is the killed version of $X$; see Section~\ref{sec:model})\vspace{-2mm}
				\item  and $\tau^{p^{(2)}} > n-1$ and $\tilde X_n \neq K$, then $\hat X_n = (X_n,C)$				\vspace{-2mm}		
			\end{itemize}		
			\item if $\hat X_{n-1} = (X_{m},S)$ for some $m<n-1$, then $\hat X_n = \hat X_{n-1}$ 	\vspace{-2mm}
			\item if $\hat X_{n-1} = K$, then $\hat X_n = K$. \vspace{-2mm}
			\defEnd
		\end{itemize} 
	\end{definition}	
	Noting that
	\begin{align*}
		\hat E_{S \cup K}:=\{ \hat x \in \hat E : \hat x = K  \mbox{ or }  \hat x = (x,S) \mbox{ for some } x \in E\}
	\end{align*}
	are absorbing states, we see that $\hat{X}$ is a Markov process on $\hat{E}$ and we denote the expectation associated to $\hat{X}_0 = \hat{x} \in \hat{E}$ by $\mathbb{E}_{\hat{x}}$. 
	The Markov kernel of $\hat{X}$, denoted by $\hat \Pi$, can be represented by 
	\begin{align}\label{markov-kernel}
		\begin{split}
			\hat{\Pi}((x,S),\hat B) &= \delta_{(x,S)}(\hat B) \\
			\hat{\Pi}(K,\hat B) &= \delta_{K}(\hat B)\\
			\hat{\Pi}((x,C), \hat B) 
			&= p^{(2)}(x) \delta_{(x,S)}(\hat B)  + (1- p^{(2)}(x))(1-\alpha)\delta_{K}(\hat B) \\
			& \enskip + (1-p^{(2)}(x)) \alpha \Pi (x, \{y \in E: (y,C) \in \hat B\}),
		\end{split}
	\end{align}
	where $\hat B \subseteq \hat E$. 
	To see that this holds, note, for example, that $\hat{\Pi}((x,C), \hat B)$ is a convex combination of 
	the measure $\delta_{(x,S)}$ (with weight $p^{(2)}(x)$, corresponding to the probability that player $2$ stops), 
	the measure $\delta_K$ (with weight $(1-p^{(2)}(x))(1-\alpha)$, corresponding to the probability that player $2$ does not stop and the process is killed),  
	and the Markov kernel for $X$, i.e., $\Pi$, (with weight $(1-p^{(2)}(x))\alpha$, corresponding to the probability that neither stopping nor killing occurs).  
	
	We are now ready to present the associated stopping problem. It is
	\begin{align}\label{sec3:optstop2}
		\hat V({\hat x}):=\sup_{\tau \in \hat {\cal T}_1} \mathbb{E}_{\hat{x}} \left[ \hat{r}(\hat{X}_\tau) \right], \quad \hat{x} \in \hat{E},
	\end{align} 
	where $\hat r: \hat E \rightarrow \mathbb{R}$ is defined by 
	\begin{align}
		\begin{split} \label{r-reward}
			\hat{r}(\hat x)  &:= 
			\begin{cases}
				(1-p^{(2)}(x))f_1(x) + p^{(2)}(x)h_1(x), & \hat x = (x,C)\\
				g_1(x), & \hat x = (x,S)\\
				0, & \hat x = K,
			\end{cases}
		\end{split}
	\end{align}
	and $\hat {\cal T}_1$ is the set of a.s. finite stopping times for the filtration $\sigma\left(\hat X_0,...,\hat X_n,\xi^{(1)}_{0},...,\xi^{(1)}_{n}\right), n \in \mathbb{N}_0$. In order to study \eqref{sec3:optstop2} we consider stopping times of the kind
	\begin{align} \label{eq:admiss-for-alt-prob}
		\hat \tau^{(p)}:= 
		\inf \{ n \in \mathbb{N}_0: p(X_n) \ge \xi^{(i)}_n\}
		\wedge \tau_{\hat E_{S \cup K}}
	\end{align}
	where $p:E\rightarrow [0,1]$ is measurable and $\tau_{\hat E_{S \cup K}}:= \inf\{n \in\mathbb{N}_0: \hat{X}_n  \in \hat E_{S \cup K} \}$ is the first entry time for $\hat{X}$ into the set of absorbing states $\hat{E}_{S \cup K}$. 
	Note that $\hat \tau^{(p)}$ is a.s. finite since $\tau_{\hat E_{S \cup K}}$ is so; see  the proof of Proposition~\ref{thm:characterization_optimal_strategies} below. 
	In particular, it holds that $\hat \tau^{(p)}\in \hat {\cal T}_1$ for any measurable function $p:E\rightarrow [0,1]$, and in Proposition~\ref{thm:characterization_optimal_strategies} we shall also see that stopping times of this kind attain the supremum in \eqref{sec3:optstop2}. 
	We also need
	\begin{align}\label{eq:indiff-cont-sets}
		\begin{split}
			\hat D &: = \{(x,C) \in \hat E: \hat{\Pi} \hat{V}(x,C) < \hat r(x,C)\}\\
			\hat I &: = \{(x,C) \in \hat E : \hat{\Pi} \hat{V}(x,C) = \hat r(x,C)\}
		\end{split}
	\end{align}
	which we interpret as 
	a \emph{(strict) stopping set} and 
	the \emph{indifference between stopping and continuing set}, 
	for problem \eqref{sec3:optstop2}, respectively, restricted to non-absorbed states $\hat x = (x,C)$. 
	It is now easily verified that 
	\begin{align*}
		\hat D\cup \hat I \cup \hat E_{S \cup K} = \{ \hat x \in \hat E: \hat{V}(\hat x) = \hat r(\hat x)\}
	\end{align*} 
	i.e., $\hat D\cup \hat I \cup \hat E_{S \cup K}$ is the stopping set -- in the usual sense of optimal stopping theory -- for the optimal stopping problem \eqref{sec3:optstop2}.

	\begin{prop}\label{thm:characterization_optimal_strategies} 
		(Ai) The first entry time
		\begin{align} \label{eq:opt-stop-problem2}
			\tau_{\hat D\cup \hat I \cup \hat E_{S \cup K}}: =
			\inf \{ n \in \mathbb{N}_0: \hat X_n \in \hat D\cup \hat I \cup \hat E_{S \cup K} \}
		\end{align}
		is an optimal stopping time for problem \eqref{sec3:optstop2}. 
		Moreover, the optimal value function $\hat{V}: \hat{E} \rightarrow \mathbb{R}$ 
		corresponding to  \eqref{sec3:optstop2} satisfies, for $\hat x \in \hat E$, the Wald-Bellman equation
		\begin{align}
			\label{eq:OptimalityEquation}
			\hat{V}(\hat{x}) = \max \{\hat{\Pi}\hat{V}(\hat{x}), \hat{r}(\hat{x})\}.
		\end{align} 
		(Aii)
		The optimal stopping time \eqref{eq:opt-stop-problem2} 
		can be represented in the form \eqref{eq:admiss-for-alt-prob} 
		with
		\begin{align*}
			\begin{split}
				\begin{cases}
					p(y) = 1, &       \text{ for all $(y,C)\in \hat D  \cup \hat I$}\\
					p(y) = 0, &        \text{ otherwise}.
				\end{cases}
			\end{split}
		\end{align*}
		(Aiii) Let $p: E \rightarrow [0,1]$ be a measurable function. The stopping time $\hat \tau^{(p)}$ (cf. \eqref{eq:admiss-for-alt-prob}) is optimal in \eqref{sec3:optstop2} for all non-absorbing initial states $\hat{x} \in \{(x,C): x \in E\}$ if and only if
		\begin{align}
			\begin{split}\label{eq:optimal-p}
				\begin{cases}
					p(y) = 1, &       \text{ for all $y\in E$ such that $(y,C)\in \hat D$}\\
					p(y) \in [0,1], & \text{ for all $y\in E$ such that $(y,C)\in \hat I$}\\
					p(y) = 0, &       \text{ otherwise}.
				\end{cases}
			\end{split}
		\end{align}
		
		(B) 
		Let $p:E \rightarrow[0,1]$ be a measurable function. 
		The stopping time $\tau^{(p)} \in \mathcal{T}_1$ (cf. \eqref{eq:M-rand-stop-time}) is optimal for the stopping problem \eqref{sec3:optstop1} with initial value 
		$x \in E$ if and only if the stopping time $\hat \tau^{(p)}\in {\hat{\mathcal{T}}}_1$ (cf. \eqref{eq:admiss-for-alt-prob}) is optimal for the stopping problem \eqref{sec3:optstop2} with initial value $\hat{x} = (x,C)$.
		Moreover, the corresponding optimal values coincide, i.e.,
		\begin{align}
			\label{eq:value_original_vs_aux}
			\sup_{\tau_1\in {\cal T}_1} J_1\left(x;\tau_1, \tau^{p^{(2)}}\right) =\sup_{\tau \in \hat {\cal T}_1} \mathbb{E}_{(x,C)} \left[ \hat{r}(\hat{X}_\tau) \right].
		\end{align}
		(Recall that the problem in the right hand side of \eqref{eq:value_original_vs_aux} depends on $p^{(2)}$ via \eqref{markov-kernel} and \eqref{r-reward}.)
	\end{prop}
	\begin{remark} 
		Note that the optimal stopping problems \eqref{sec3:optstop1} and \eqref{sec3:optstop2} are equivalent (in the sense of Proposition~\ref{thm:characterization_optimal_strategies}(B)) only when the initial state in the latter problem is non-absorbing, i.e., when $\hat x = (x,C)$ for some $x\in E$. 
		\defEnd
	\end{remark}

	\begin{proof}[Proof of Proposition~\ref{thm:characterization_optimal_strategies}]
		(Ai) Note  that $\hat{X}$ is a time-homogeneous Markov process on $\hat{E}$. 
		Writing $k_1=\max\{|f_1|,|g_1|,|h_1|\}$, it is clear by construction (Definition~\ref{def:X-hat}) that $|\hat{r}(\hat{X}_n)|\leq \max_{m\leq n}k_1(\tilde{X}_m)\leq M$, so that by Assumption~\ref{assum:bounded-functions} we have
		\begin{align}
			\label{eq:bounded_transf}
			\mathbb{E}_{x}\left[ \sup_{n \in \mathbb{N}_0} |\hat{r}(\hat{X}_n)| \right]\leq \mathbb{E}_{x}\left[ \sup_{n \in \mathbb{N}_0} |k_1(\tilde{X}_n)| \right]<\mathbb{E}_{x}\left[ M \right] <\infty.
		\end{align}
		Moreover, the stopping time \eqref{eq:opt-stop-problem2}	is a.s. finite -- this is verified using the killing feature of $\hat X$, which implies that 
		$
		\mathbb{P}_x(\tau_{\hat E_{S \cup K}} = \infty) \le \lim_{n \rightarrow \infty} \alpha^n = 0.
		$
		Hence, the result follows from standard optimal stopping theory (see, e.g., \cite[Theorem 1.11]{peskir2006optimal}). 
		
		(Aii) This result is directly verified.

		(Aiii) 
		Let us first assume that $\hat \tau^{(p)}$ is optimal  in \eqref{sec3:optstop2} and show that this implies that $p$ satisfies \eqref{eq:optimal-p}. The second part of \eqref{eq:optimal-p} is trivially satisfied. Moreover, it follows from (Ai) and standard optimal stopping theory 
		(ibid.) 
		that $\tau_{\hat D  \cup \hat I \cup \hat E_{S \cup K}} \le \hat \tau^{(p)}$ a.s. 
		From this  (as well as (Ai)-(Aii)) it directly follows that $p$ satisfies the last part of \eqref{eq:optimal-p}.

		Now assume (to obtain a contradiction) that $p$ does not satisfy the first part of \eqref{eq:optimal-p}, i.e., $p(x_0)<1$ for some $x_0 \in E$ such that $(x_0,C) \in \hat D$. 
		Using basic observations and the definition of $\hat D$ in \eqref{eq:indiff-cont-sets} we obtain
		\begin{align*}
			\mathbb{E}_{(x_0,C)} \left[ \hat{r}(\hat{X}_{\hat \tau^{(p)}}) \right] 
			&= p(x_0) \hat{r}(x_0,C) + (1-p(x_0)) \mathbb{E}_{(x_0,C)} \left[ \mathbb{E}_{\hat{X}_1} \left[ \hat{r}(\hat{X}_{\hat \tau^{(p)}})  \right] \right] \\
			&\le p(x_0) \hat{V}(x_0,C) + (1-p(x_0)) \mathbb{E}_{(x_0,C)} \left[ V(\hat{X}_1) \right] \\
			&=  p(x_0) \hat{V}(x_0,C) + (1-p(x_0))\hat{\Pi} \hat{V}(x_0,C) < \hat{V}(x_0,C).
		\end{align*}
		Hence, $\hat \tau^{(p)}$ cannot be optimal for $(x_0,C)$ and we have a contradiction. It follows that $p$ satisfies \eqref{eq:optimal-p}. 
		
		Let us now assume that $p$ satisfies \eqref{eq:optimal-p} and show that the corresponding stopping time $\hat \tau^{(p)}$ 
		is optimal for \eqref{sec3:optstop2} for any initial value $\hat x = (x, C)$. 
		Let $\tilde{p}: E \rightarrow [0,1]$ be defined so that it corresponds to \eqref{eq:opt-stop-problem2}, i.e., as in (Aii). 
		For $N \in \mathbb{N}_0$ let
		\begin{align*}
			\tau_N:= \inf \big\{n \in \mathbb{N}_0: \big( p(X_n) \ge \xi_n^{(1)} \mbox{ and } n < N \big)  \mbox{ or }  \big(\tilde p(X_n) \ge \xi_n^{(1)} \mbox{ and } n \geq N \big) \big\} \wedge \tau_{\hat E_{S \cup K}}
		\end{align*}
		which corresponds to using the strategy $p$ before time $N$ and using the strategy $\tilde p$ from time $N$ and onwards (until absorption). 
		Note that $\tau_0$ is then equal to the stopping time \eqref{eq:opt-stop-problem2}, and $\tau_0$ is therefore   optimal in \eqref{sec3:optstop2}. 
		We now show that each stopping time $\tau_N, N \in \mathbb{N}_0,$ is optimal. Note that 
		\begin{align*}
			&\mathbb{E}_{(x,C)} \left[ \hat{r}(\hat{X}_{\tau_{N+1}})\right] \\ 
			&= \mathbb{E}_{(x,C)} \left[ \sum_{n=0}^\infty \mathbb{I}_{\{\tau_{N+1} = n\}} \hat{r}(\hat{X}_n) \right] \\
			&= \mathbb{E}_{(x,C)} \left[ \sum_{n=0}^N \mathbb{I}_{\{\tau_{N+1} = n\}} \hat{r}(\hat{X}_n) + \mathbb{I}_{\{\tau_{N+1} = N+1\}} \hat{r}(\hat{X}_{N+1}) + \sum_{n=N+2}^\infty \mathbb{I}_{\{\tau_{N+1} = n\}} \hat{r}(\hat{X}_n) \right] \\
			&= \mathbb{E}_{(x,C)} \left[ \sum_{n=0}^N \mathbb{I}_{\{\tau_{N+1} = n\}} \hat{r}(\hat{X}_n) + \mathbb{I}_{\{\tau_{N+1} >N\}}  \mathbb{I}_{\{\hat{X}_{N+1} \in \hat E_{S \cup K}\}} \hat{r}(\hat{X}_{N+1})\right. \\
			&\quad + \mathbb{I}_{\{\tau_{N+1} >N\}} \mathbb{I}_{\{\hat{X}_{N+1} \notin \hat E_{S \cup K}\}} \mathbb{I}_{\{p(X_{n+1}) \ge \xi_{N+1}^{(1)} \}}  \hat{r}(\hat{X}_{N+1}) \\
			&\quad \left. + \mathbb{I}_{\{\tau_{N+1}> N\}} \mathbb{I}_{\{\hat{X}_{N+1} \notin \hat E_{S \cup K}\}} \mathbb{I}_{\{p(X_{n+1}) < \xi_{N+1}^{(1)} \}} \mathbb{E}_{\hat{X}_{N+2}}\left[ \sum_{n=0}^\infty \mathbb{I}_{\{\tau_{\hat{D} \cup \hat{I} \cup \hat{E}_{S \cup K}} = n\}} \hat{r}(\hat{X}_n) \right] \right] \\
			&= \mathbb{E}_{(x,C)} \left[ \sum_{n=0}^N \mathbb{I}_{\{\tau_{N} = n\}} \hat{r}(\hat{X}_n) + \mathbb{I}_{\{\tau_{N} >N\}}  \mathbb{I}_{\{\hat{X}_{N+1} \in \hat E_{S \cup K}\}} \hat{r}(\hat{X}_{N+1}) \right.  \\
			& \quad + \left. \mathbb{I}_{\{\tau_{N} >N\}} \mathbb{I}_{\{\hat{X}_{N+1} \notin \hat E_{S \cup K}\}}  p(X_{N+1}) \hat{r}(\hat{X}_{N+1}) + \mathbb{I}_{\{\tau_N> N\}} \mathbb{I}_{\{\hat{X}_{N+1} \notin \hat E_{S \cup K}\}}  (1- p(X_{N+1})) \Pi \hat{V}(\hat{X}_{N+1})  \right]. 
		\end{align*}
		Since $p$ satisfies \eqref{eq:optimal-p} we have that $p(x)=1$ implies $\hat{V}(x,C)= \hat{r}(x,C)$, $p(x)\in (0,1)$ implies $\Pi \hat{V}(x,C) = \hat{r}(x,C) = \hat{V}(x,C)$, and $p(x)=0$ implies $\hat{V}(x,C)= \hat{\Pi} \hat{V}(x,C)$.  Hence,  
		\[
		p(x) \hat{r}((x,C)) + (1- p (x)) \Pi \hat{V}(x,C) = \hat{V}(x,C)
		\] for all $x \in E$. Moreover, it holds that 
		$\hat{V}(\hat{x}) = \hat{r}(\hat{x})$ for each $\hat{x} \in \hat E_{S \cup K}$ since these states are absorbing. Putting the parts above together yields
		\begin{align*}
			&\mathbb{E}_{(x,C)} \left[ \hat{r}(\hat{X}_{\tau_{N+1}})\right] \\
			&= \mathbb{E}_{(x,C)} \left[ \sum_{n=0}^N \mathbb{I}_{\{\tau_{N} = n\}} \hat{r}(\hat{X}_n) +\mathbb{I}_{\{\tau_{N} >N\}}  \mathbb{I}_{\{\hat{X}_{N+1} \in \hat E_{S \cup K}\}} \hat{V}(\hat{X}_{N+1}) +  \mathbb{I}_{\{\tau_{N} >N\}} \mathbb{I}_{\{\hat{X}_{N+1} \notin \hat E_{S \cup K}\}} \hat{V}(\hat{X}_{N+1})  \right] \\
			&= \mathbb{E}_{(x,C)} \left[ \sum_{n=0}^N \mathbb{I}_{\{\tau_{N} = n\}} \hat{r}(\hat{X}_n) + \mathbb{I}_{\{\tau_{N} >N\}}  \hat{V}(\hat{X}_{N+1})  \right] \\
			&= \mathbb{E}_{(x,C)} \left[ \sum_{n=0}^N \mathbb{I}_{\{\tau_{N} = n\}} \hat{r}(\hat{X}_n)  + \sum_{n=N+1}^\infty \mathbb{I}_{\{\tau_{N} = n\}} \hat{r}(\hat{X}_n) \right] \\
			&= \mathbb{E}_{(x,C)} \left[ \hat{r}(\hat{X}_{\tau_{N}})\right].
		\end{align*}
		Hence, $
		\mathbb{E}_{(x,C)} \left[ \hat{r}(\hat{X}_{\tau_N}) \right] = \mathbb{E}_{(x,C)} \left[ \hat{r}(\hat{X}_{\tau_0}) \right]$, i.e., each $\tau_N$ is optimal. 
		Now note that  $\tau_N \rightarrow \tau^{(p)}$  a.s., as $N \rightarrow \infty$. 
		Hence, since $|\hat{r}(\hat{X}_{\tau_N})| \leq M$ (cf. the beginning of the proof), we  obtain 
		$
		\mathbb{E}_{(x,C)} \left[ \hat{r}(\hat{X}_{\tau^{(p)}}) \right] = \mathbb{E}_{(x,C)} \left[ \hat{r}(\hat{X}_{\tau_0}) \right],
		$ 
		which shows that also $\tau^{(p)}$ is optimal.

		(B) The states in $\hat E_{S \cup K}$ are absorbing, so that $\mathbb{E}_{(x,C)}\left[ \hat{r}(\hat{X}_{\hat\tau_1}) \right]  = \mathbb{E}_{(x,C)}\left[ \hat{r}\left(\hat{X}_{\hat\tau_1 \wedge \tau_{\hat{E}_{S \cup K}}}\right) \right]$, for $\hat \tau_1 \in \hat{\mathcal{T}}_1$. 
		Hence, it suffices to consider stopping times of the form $\hat{\tau_1} \wedge  \tau_{\hat E_{S \cup K}}$ in the supremum in the right-hand side of \eqref{eq:value_original_vs_aux}. 
		Moreover, for $\tau_1 \in \mathcal{T}_1$, it is directly seen that $\tau_1 \wedge \tau_{\hat{E}_{S \cup K}} \in \hat{\mathcal{T}_1}$. 
		Using that as long as $\hat{X}_n$ is not absorbed we have $\hat{X}_n = (X_n,C)$, we see that we can for any $\hat{\tau}_1  \in \hat{\mathcal{T}_1}$ find a stopping time $\tau_1 \in \mathcal{T}_1$ such that $\tau_1 \wedge \tau_{\hat E_{S \cup K}} = \hat{\tau}_1 \wedge \tau_{\hat E_{S \cup K}}$. 
		Hence, it actually suffices to consider stopping times $\tau_1 \wedge \tau_{\hat E_{S \cup K}}$ with $\tau_1 \in \mathcal{T}_1$ in the supremum in the right-hand side of \eqref{eq:value_original_vs_aux}. In the following we therefore prove that 
		\begin{align}
			\label{eq:proofSameReward}
			J_1\left(x;\tau_1, \tau^{p^{(2)}}\right)
			= \mathbb{E}_{(x,C)} \left[ \hat{r}\left(\hat{X}_{\tau_1 \wedge \tau_{\hat E_{S \cup K}}}\right) \right]
		\end{align}
		for all $x \in E$ and all stopping times $\tau_1 \in \mathcal{T}_1$. This in turn implies that for any starting value $x \in E$, 
		a stopping time is optimal in \eqref{sec3:optstop1} if and only if it is also optimal for the stopping problem \eqref{sec3:optstop2} for the starting value $(x,C)$ and the result follows.

		Note, in order to prove \eqref{eq:proofSameReward} it suffices to show, for all $n \in \mathbb{N}_0$ and $x \in E$, that 
		\begin{align}\label{eq:proofSameReward-2}
			\begin{split}
				& \mathbb{E}_x\left[\alpha^n f_1(X_n) \mathbb{I}_{\{n < \tau_2\}} + \alpha^{\tau_2} g_1(X_{\tau_2}) \mathbb{I}_{\{\tau_2 < n\}} + \alpha^n h_1(X_{n}) \mathbb{I}_{\{n= \tau_2\}} | X_1, \ldots, X_n \right]\\
				&=\mathbb{E}_{(x,C)}\left[ \hat{r}(\hat{X}_n) |X_1, \ldots, X_n \right],
			\end{split}
		\end{align}
where we have written $\tau_2= \tau^{p^{(2)}}$. It is directly seen that
		\begin{align*}
			\mathbb{E}_{(x,C)}\left[ \hat{r}((x,C))\right] &= (1-p^{(2)}(x)) f_1(x) + p^{(2)}(x) h_1(x) \\
			&= \mathbb{E}_x\left[ \mathbb{I}_{\{p_2(x) < \xi_0^{(2)}\}} f_1(x) + \mathbb{I}_{\{p_2(x) \ge \xi_0^{(2)}\}} h_1(x)\right] \\
			&= \mathbb{E}_x\left[ f_1(x)\mathbb{I}_{\{0<\tau_2\}}  + h_1(x) \mathbb{I}_{\{\tau_2=0\}} \right],
		\end{align*}
		which is \eqref{eq:proofSameReward-2} for $n=0$ (to see this recall e.g., \eqref{r-reward} and that $\hat{X}_0=(x,C)$ and $X_0=x$). 
		For $n \ge 1$, it can be verified -- using e.g., the definitions of $\hat X$ and $\tilde X$, as well as \eqref{eq:M-rand-stop-time} and \eqref{r-reward} -- that 
		\begin{align}
			\notag
			&\mathbb{E}_{(x,C)}[ \hat{r} (\hat{X}_n)|X_1, \ldots, X_n] \\ \notag
			&= \sum_{i=1}^n \left( \prod_{j=0}^{i-2} (1-p^{(2)}(X_j)) \right) p^{(2)}(X_{i-1}) \alpha^{i-1} g_1(X_{i-1}) \\ \notag
			&\quad  + \left( \prod_{j=0}^{n-1} (1- p^{(2)}(X_j)) \right) \alpha^n \left( (1-p^{(2)}(X_n)) f_1(X_n) + p^{(2)}(X_n) h_1(X_n) \right) \\
			\label{eq:RelationRewardGameStopping}
			\begin{split}
				&= \sum_{i=1}^n \left( \prod_{j=0}^{i-2} (1-p^{(2)}(X_j)) \right) p^{(2)}(X_{i-1}) \alpha^{i-1} g_1(X_{i-1}) \\
				&\quad + \left( \prod_{j=0}^{n} (1- p^{(2)}(X_j)) \right) \alpha^n f_1(X_n)  + \left( \prod_{j=0}^{n-1} (1- p^{(2)}(X_j)) \right) p^{(2)}(X_n) \alpha^n h_1(X_n) 
			\end{split} \\ \notag
			&= \mathbb{E}_x \left[ \mathbb{I}_{\{\tau_2 <n\}} \alpha^{\tau_2} g_1(X_{\tau_2}) + \mathbb{I}_{\{\tau_2 >n\}} \alpha^n f_1(X_n) + \mathbb{I}_{\{\tau_2 = n\}} \alpha^n h_1(X_n) \Big| X_1, \ldots, X_n \right], 
		\end{align}
		which means that \eqref{eq:proofSameReward-2} holds also for $n \geq 1$.
	\end{proof}
	We are now in a position to prove Proposition~\ref{thm:CharacterizationValues}.

	\begin{proof}[Proof of Proposition~\ref{thm:CharacterizationValues}]
		It follows from Proposition~\ref{thm:characterization_optimal_strategies}(Aiii)--(B) that $\text{BR}^{(1)}\left(p^{(2)}\right)$ is the set of all measurable functions
		$p:E \rightarrow [0,1]$ that satisfy \eqref{eq:optimal-p}. 
		By inspection of \eqref{eq:optimal-p} we find that this set is non-empty, closed and convex.
	\end{proof}

	\section{Verification and Characterization for the General Game Formulation}\label{sec:ver-thm}
	Suppose $\left(p^{(1)},p^{(2)}\right)$ is a global Markovian (randomized) equilibrium.  The corresponding (equilibrium) values are defined as the function pair $\left(V^{(1)}, V^{(2)}\right)$ with
	\[
	V^{(i)}(x):= J_i\left(x,\tau^{p^{(1)}}, \tau^{p^{(2)}}\right), \enskip x\in E.
	\] 
	In this section we will characterize $\left(p^{(1)},p^{(2)}\right)$ and $\left(V^{(1)}, V^{(2)}\right)$ as a solution to the system:
	\begin{subequations}
		\label{eq:verification_all}
		\begin{align}
			\label{eq:verification_value1}
			V^{(1)}(x) &= \max \left\{ 
			(1-p^{(2)}(x)) \alpha \Pi V^{(1)}(x) + p^{(2)}(x) g_1(x), 
			(1-p^{(2)}(x)) f_1(x) + p^{(2)}(x) h_1(x)
			\right\} \\
			\label{eq:verification_value2}
			V^{(2)}(x) &=\max \left\{ 
			(1-p^{(1)}(x)) \alpha \Pi V^{(2)}(x) + p^{(1)}(x) g_2(x), 
			(1-p^{(1)}(x)) f_2(x) + p^{(1)}(x) h_2(x)
			\right\}\\
			\label{eq:verification_stop1}
			p^{(1)}(x)&>0 \Rightarrow (1-p^{(2)}(x)) \alpha \Pi V^{(1)}(x) + p^{(2)}(x) g_1(x) \le  (1-p^{(2)}(x)) f_1(x) + p^{(2)}(x) h_1(x) \\
			\label{eq:verification_cont1}
			p^{(1)}(x)&<1 \Rightarrow (1-p^{(2)}(x)) \alpha \Pi V^{(1)}(x) + p^{(2)}(x) g_1(x) \ge  (1-p^{(2)}(x)) f_1(x) + p^{(2)}(x) h_1(x) \\
			\label{eq:verification_stop2}
			p^{(2)}(x)&>0 \Rightarrow (1-p^{(1)}(x)) \alpha \Pi V^{(2)}(x) + p^{(1)}(x) g_2(x) \le (1-p^{(1)}(x)) f_2(x) + p^{(1)}(x) h_2(x) \\
			\label{eq:verification_cont2}
			p^{(2)}(x)&<1 \Rightarrow (1-p^{(1)}(x)) \alpha \Pi V^{(2)}(x) + p^{(1)}(x) g_2(x) \ge (1-p^{(1)}(x)) f_2(x) + p^{(1)}(x) h_2(x)
		\end{align}
	\end{subequations}
	\begin{remark} Let us interpret the system \eqref{eq:verification_all} from the view-point of player $1$. The interpretation for player $2$ is analogous. 
		The interpretation of \eqref{eq:verification_value1} is that the equilibrium value is equal to the maximum of the (expected) values corresponding to 
		stopping with probability zero (first part of the maximum function) and
		stopping with probability one (second part of the maximum function). 
		The interpretation of \eqref{eq:verification_stop1} is that if we assign a positive probability to stopping then the value of stopping must dominate (at least weakly) that of continuing. The interpretation of \eqref{eq:verification_cont1} is analogous.
		Note that \eqref{eq:verification_stop1} and \eqref{eq:verification_cont1} imply that
		\begin{align*}
			(1-p^{(2)}(x)) \alpha \Pi V^{(1)}(x) + p^{(2)}(x) g_1(x) =  (1-p^{(2)}(x)) f_1(x) + p^{(2)}(x) h_1(x)
		\end{align*}
		whenever $p^{(1)}(x) \in (0,1)$. This corresponds to the game theory indifference principle, which for our game means that a player will randomize in equilibrium at a given state $x$ only if the player is indifferent between stopping and not stopping at that state. 
		\defEnd
	\end{remark}

	\begin{theorem}[Equilibrium characterization and verification]
		\label{thm:veri}
		Let $p^{(1)},p^{(2)}: E \rightarrow [0,1]$, $V^{(1)}, V^{(2)}: E \rightarrow \mathbb{R}$ be measurable functions. 
		(A): Suppose $\left(p^{(1)}, p^{(2)}\right)$ is a global Markovian randomized equilibrium with values $\left(V^{(1)}, V^{(2)}\right)$. 
		Then these functions satisfy the system \eqref{eq:verification_all}. 
		(B): Suppose $p^{(1)},p^{(2)}, V^{(1)}$, and $V^{(2)}$ satisfy the system \eqref{eq:verification_all} as well as
		\begin{equation}
			\label{eq:veri_int}
			V^{(i)}(\tilde{X}_1) \in L^1	\enskip \text{and}\enskip 	\enskip \mathbb{E}_{x} \left[\sup_{n \in \mathbb{N}_0} V^{(i)}(\tilde{X}_n) \right] < \infty, \enskip x \in E, \quad i=1,2.
		\end{equation}
		Then $\left(p^{(1)}, p^{(2)}\right)$ is a global Markovian randomized equilibrium with values $\left(V^{(1)}, V^{(2)}\right)$.
	\end{theorem}

	\begin{proof}

		(A) The analysis in Section~\ref{sec:Optimization} was without loss of generality performed from the view-point of player $1$. 
		Based on this analysis we will show that $\left(p^{(1)}, p^{(2)}\right)$ and $V^{(1)}$ satisfy \eqref{eq:verification_value1}, \eqref{eq:verification_stop1} and \eqref{eq:verification_cont1}. Analogous arguments can be used for \eqref{eq:verification_value2}, \eqref{eq:verification_stop2} and \eqref{eq:verification_cont2}. We set $V= V^{(1)}$ in this proof. By definition of equilibrium we have that 
		\begin{align}\label{eq:veri_aux00KL}
			V(x)
			& = J_1\left(x;\tau^{p^{(1)}}, \tau^{p^{(2)}}\right) =\sup_{\tau_1\in {\cal T}_1} J_1\left(x;\tau_1, \tau^{p^{(2)}}\right).
		\end{align}
		The stopping problems \eqref{sec3:optstop1} and \eqref{sec3:optstop2} are equivalent (in the sense of Proposition~\ref{thm:characterization_optimal_strategies}) and the states $(x,S),x \in E$, and $K$ are absorbing. 
		Hence,
		\begin{equation}
			\label{eq:veri_vhat}
			\hat{V}(\hat{x}) = \begin{cases}
				V(x) &\text{if } \hat{x} = (x,C) \\
				g_1(x) &\text{if } \hat{x} = (x,S) \\
				0 &\text{if } \hat{x} = K
			\end{cases}.
		\end{equation}
		(Recall that $\hat{V}$ is defined in \eqref{sec3:optstop2}, which depends on $p^{(2)}$ through \eqref{markov-kernel} and \eqref{r-reward}.) 
		Using  the relation between $\hat \Pi $ and $\Pi$ (cf. \eqref{markov-kernel}), 
		the definition of $\hat{r}$ (cf. \eqref{r-reward})
		and \eqref{eq:veri_vhat} we obtain 
		\begin{align}	
			\hat{\Pi} \hat{V}(x,C) & = (1-p^{(2)}(x)) \alpha \Pi V(x) + p^{(2)}(x)g_1(x)   \label{eq:veri_aux1}\\
			\hat r(x,C) &  = (1-p^{(2)}(x)) f_1(x) + p^{(2)}(x) h_1(x).\label{eq:veri_aux2}
		\end{align}
		Using the above and that
		$\hat{V}$ satisfies \eqref{eq:OptimalityEquation} (in particular for each $\hat x = (x,C)$) we find 
		\begin{align*}
			V(x) 
			& = \hat{V}(x,C)\\ 
			& = \max\left\{ \hat \Pi \hat{V}(x,C), \hat{r}(x,C)\right\}\\
			& = \max\left\{(1-p^{(2)}(x)) \alpha \Pi V(x) + p^{(2)}(x)g_1(x), (1-p^{(2)}(x)) f_1(x) + p^{(2)}(x) h_1(x)\right\}.
		\end{align*}
		Hence \eqref{eq:verification_value1} holds. 
		By assumption $\left(p^{(1)},p^{(2)}\right)$ is an equilibrium and it hence holds that $p^{(1)}$ is optimal in the sense of \eqref{eq:veri_aux00KL}. 
		Hence $p^{(1)}$ is also optimal in \eqref{sec3:optstop2} (in the sense of Proposition~\ref{thm:characterization_optimal_strategies}(B)). 
		Consider $x\in E$ with $p^{(1)}(x)>0$, then $(x,C) \in \hat D \cup \hat I$ (by Proposition~\ref{thm:characterization_optimal_strategies}(Aiii)) which implies (by \eqref{eq:indiff-cont-sets}) that $\hat{\Pi} \hat{V}(x,C) \le \hat r(x,C)$. 
		This implies (using \eqref{eq:veri_aux1}--\eqref{eq:veri_aux2}) that \eqref{eq:verification_stop1} holds. 
		\eqref{eq:verification_cont1} can be verified using similar arguments.  
		
		(B) The functions $p^{(1)},p^{(2)}, V^{(1)}$ and $V^{(2)}$ satisfy \eqref{eq:verification_all} by assumption. 
		With this in mind we will below use the notation of Section~\ref{sec:Optimization} to show 
		(i) that $V^{1}$ is equal to the supremum $V$ in \eqref{eq:veri_aux00KL}
		and 
		(ii)
		that this supremum is attained by $p^{(1)}$, for any $x\in E$. 
		The analogous statement can be proved for $V^{2}$ and $p^{(2)}$. 
		By definition of global Markovian randomized equilibrium this concludes the proof of (B).
		To this end we here \emph{define} the function $\hat{V}$ according to 
		\begin{equation} \label{eq:veri_vhat2:kl}
			\hat{V}(\hat{x}):= \begin{cases}
				V^{(1)}(x) &\text{if } \hat{x} = (x,C) \\
				g_1(x) &\text{if } \hat{x} = (x,S) \\
				0 &\text{if } \hat{x} = K
			\end{cases}.
		\end{equation}
		Let us first prove that $\hat V$ solves the Wald-Bellman equation \eqref{eq:OptimalityEquation} for all $x \in \hat{E}$. 
		For all absorbed states $\hat x = (x,S)$ and $\hat x = K$, this follows immediately from the definition \eqref{eq:veri_vhat2:kl} and \eqref{r-reward}.
		For $x \in E$ we note, again by the definition \eqref{eq:veri_vhat2:kl}, that $\hat{V}$ also satisfies  \eqref{eq:veri_aux1}. 
		Thus, combining \eqref{eq:veri_aux1} and \eqref{eq:veri_aux2}, and using that $V^{(1)}$ satisfies 
		\eqref{eq:verification_value1} we immediately see that $\hat V$ solves the Wald-Bellman equation also for $\hat x = (x,C)$.
		Moreover, by definition of $\hat{V}$, 
		Assumption~\ref{assum:bounded-functions} and \eqref{eq:veri_int} we find that
		\begin{align*}
			\mathbb{E}_{\hat{x}} \left[ \sup_{n \in \mathbb{N}_0} \hat{V}(\hat{X}_n) \right] 
			&\le \mathbb{E}_{\hat{x}} \left[ \sup_{n \in \mathbb{N}_0} \max \left\{ V^{(1)}(\tilde{X}_n), g_1(\tilde{X}_n), 0 \right\} \right]  < \infty.
		\end{align*} 
		It can be similarly shown that $\hat{V}(\hat{X}_1) \in L^1$.
		Finally, using that $\hat{X}$ is absorbed in finite time a.s. and that $\hat{V}(\hat{x})=\hat{r}(\hat{x})$ for all absorbing states $\hat{x} \in \hat E_{S \cup K}$ we obtain
		$\lim_{n \rightarrow \infty} \hat{V}(\hat{X}_n) = \lim_{n \rightarrow \infty} \hat{r}(\hat{X}_n)$ a.s. All these observations lead us to conclude 
		-- based on standard stopping theory, see, e.g., \cite[Theorem 1.13]{peskir2006optimal} --
		that $\hat V$ is equal to the optimal value in \eqref{sec3:optstop2}, which in turn implies that item (i) above is true 
		(to see this use \eqref{eq:veri_vhat2:kl} and Proposition~\ref{thm:characterization_optimal_strategies}(B)).

		All we have left is to prove item (ii) above.
		First note (using \eqref{eq:verification_stop1}--\eqref{eq:verification_cont1} and \eqref{eq:veri_aux1}--\eqref{eq:veri_aux2}) that $p^{(1)}$ satisfies: 
		$p^{(1)}(x) > 0 \Rightarrow \hat \Pi \hat{V}(x,C)  {\leq} \hat r(x,C) \Rightarrow x \in \hat D \cup \hat I$ and 
		$p^{(1)}(x) < 1 \Rightarrow \hat \Pi  \hat{V}(x,C)  {\geq} \hat r(x,C) \Rightarrow x \notin \hat D$ (recall \eqref{eq:indiff-cont-sets}). 
		But this in turn implies that $p^{(1)}$ and $\hat V$ satisfies \eqref{eq:optimal-p}. It hence follows (cf. Proposition~\ref{thm:characterization_optimal_strategies}(Aiii)-(B))  that (ii) is true.  
	\end{proof}

	\section{Results for Zero-Sum Games}\label{sec:zero-sum}
	In this section we study the zero-sum game specification.  
	The main results are as follows.  
    In Section~\ref{sec:zero_constr}, we show that any zero-sum game satisfying our integrability assumption admits a global Markovian equilibrium, and we provide an explicit construction of such an equilibrium.
	In Section~\ref{sec:zero_random} we obtain both sufficient and necessary conditions for the non-existence of a global Markovian pure equilibrium.
	In Section~\ref{sec:zerosum-randomNE-no-mix-example}, we present an explicit example in which a global
	Markovian randomized equilibrium but no global
	Markovian pure equilibrium exists.

	We obtain the usual zero-sum game in our setting in case the payoff functions satisfy
	\[
	f_1=f, \enskip 
	g_1 = g, \enskip 
	h_1 = h, \enskip 
	f_2 = -g, \enskip 
	g_2=-f, \enskip 
	h_2 = -h,
	\]
	with $f,g,h: E \rightarrow \mathbb{R}$. This implies that the reward of one player is the loss of the other player and it is easily seen that the game can be viewed as follows: player $1$ chooses a strategy $\tau_1$ to maximize the expected value
	\begin{equation*}
		\mathbb{E}_x \left[ \alpha^{\tau_1} f(X_{\tau_1}) \mathbb{I}_{\{\tau_1<\tau_2\}} + \alpha^{\tau_2} g(X_{\tau_2}) \mathbb{I}_{\{\tau_2 < \tau_1\}} + \alpha^{\tau_1} h(X_{\tau_1}) \mathbb{I}_{\{\tau_1 = \tau_2 < \infty\}} \right],
	\end{equation*} 
	while player $2$ chooses a strategy $\tau_2$ to \emph{minimize} it. 
	Given a global Markovian (possibly randomized) equilibrium  $\left(p^{(1)}, p^{(2)}\right)$, and writing $\tau_i = \tau^{p^{(i)}},i=1,2$, we identify 
	the equilibrium value with $V: E \rightarrow \mathbb{R}$ defined by
	\begin{align}
		\label{eq:zero_value_fct}
		V(x)  = \mathbb{E}_x \left[ \alpha^{\tau_1} f(X_{\tau_1}) \mathbb{I}_{\{\tau_1<\tau_2\}} + \alpha^{\tau_2} g(X_{\tau_2}) \mathbb{I}_{\{\tau_2 < \tau_1\}} + \alpha^{\tau_1} h(X_{\tau_1}) \mathbb{I}_{\{\tau_1 = \tau_2 < \infty\}} \right].
	\end{align} 
(In the terminology of Section~\ref{sec:ver-thm}, $V$ is the value for player $1$ whereas $-V$ is the value for player $2$.)
	For this game we immediately find that Theorem~\ref{thm:veri} corresponds to:
	\begin{corollary}[Equilibrium characterization and verification]
		\label{cor:veri_zero_sum}
		\begin{subequations}
			Let $p^{(1)},p^{(2)}: E \rightarrow [0,1]$ and $V: E \rightarrow \mathbb{R}$ be measurable functions. 
			(A): Suppose   $\left(p^{(1)}, p^{(2)}\right)$  is a global Markovian randomized equilibrium with value $V$. Then these functions satisfy the system:
			\label{eq:Zero_All_Conditions}
			\begin{align}
				\label{eq:Zero_v_Maxi}
				V(x) &= \max \left\{(1-p^{(2)}(x)) \alpha \Pi V(x) + p^{(2)}(x) g(x), (1- p^{(2)}(x)) f(x) + p^{(2)}(x) h(x) \right\} \\
				\label{eq:Zero_v_Mini}
				V(x) &= \min \left\{(1- p^{(1)}(x)) \alpha \Pi V(x) + p^{(1)}(x) f(x), (1-p^{(1)}(x)) g(x) + p^{(1)}(x) h(x)\right\}\\
				\label{eq:Zero_Maxi_Stop}
				p^{(1)}(x) &>0 \Rightarrow (1-p^{(2)}(x)) \alpha \Pi V(x) + p^{(2)}(x) g(x)  \le (1- p^{(2)}(x)) f(x) + p^{(2)}(x) h(x) \\
				\label{eq:Zero_Maxi_Cont}
				p^{(1)}(x) &<1 \Rightarrow (1-p^{(2)}(x)) \alpha \Pi V(x) + p^{(2)}(x) g(x) \ge (1- p^{(2)}(x)) f(x) + p^{(2)}(x) h(x)  \\
				\label{eq:Zero_Mini_Stop}
				p^{(2)}(x) &>0 \Rightarrow	(1- p^{(1)}(x)) \alpha \Pi V(x) + p^{(1)}(x) f(x)	\ge (1-p^{(1)}(x)) g(x) + p^{(1)}(x) h(x)  \\
				\label{eq:Zero_Mini_Cont}
				p^{(2)}(x) &<1 \Rightarrow	(1- p^{(1)}(x)) \alpha \Pi V(x) + p^{(1)}(x) f(x) \le (1-p^{(1)}(x)) g(x) + p^{(1)}(x) h(x). 		
			\end{align}
		\end{subequations}
		(B): Suppose $p^{(1)},p^{(2)}$ and $V$ satisfy \eqref{eq:Zero_All_Conditions} 
		and that $\sup_{n \in \mathbb{N}_0} |V(\tilde{X}_n)| \in L^1$.
		Then $\left(p^{(1)},p^{(2)}\right)$ is a global Markovian randomized equilibrium with value $V$.
	\end{corollary}
	
	\begin{remark}
		\label{rem:value}
		Whenever an equilibrium exists, the definition of a value $V$ coincides, as expected, with the classical definition of a value in the theory of zero-sum games (as in e.g., \cite{OhtsuboTerminating,RosenbergSolanVieilleZeroSum}). In  this literature, the lower and upper value are defined as 
		\[
			\underline{V}(x) := \sup_{\tau_1 \in \mathcal{T}_1} \inf_{\tau_2 \in \mathcal{T}_2} J_1(x; \tau_1, \tau_2) \quad \text{and} \quad \overline{V}(x) = \inf_{\tau_2 \in \mathcal{T}_2} \sup_{\tau_1 \in \mathcal{T}_1} J_1(x; \tau_1, \tau_2),
		\] respectively, and it is immediate that $\underline{V}(x) \le \overline{V}(x)$. If the two values coincide for all $x \in E$, then we call $\tilde{V}: E \rightarrow \mathbb{R}$, $x \mapsto \underline{V}(x)=\overline{V}(x)$ the value.
		Now, let $(p^{(1)},p^{(2)})$ be a global Markovian equilibrium with value $V$, then, writing $\tilde\tau_i = \tau^{p^{(i)}}$, we have for any $x \in E$ and any pair $(\tau_1,\tau_2)$ of strategies
		\[
		J_1(x;\tau_1, \tilde{\tau}_2) \le J_1(x; \tilde{\tau}_1, \tilde{\tau}_2) \le J_1(x; \tilde{\tau}_1, \tau_2).
		\] Hence, we obtain
		\[
		\overline{V}(x) \le \sup_{\tau_1 \in \mathcal{T}_1} J_1(x; \tau_1, \tilde{\tau}_2) \le J_1(x; \tilde{\tau}_1, \tilde{\tau}_2) = V(x) \le \inf_{\tau_2 \in \mathcal{T}_2} J_1(x;\tilde{\tau}_1, \tau_2) \le \underline{V}(x)
		\] for all $x \in E$, which means that $\tilde{V}=V$. In particular, we obtain that every equilibrium of a zero-sum game has the same value $V$. \defEnd
	\end{remark}

    \subsection{Existence and Construction of a Global Markovian Equilibrium for General Zero-sum Games}
    \label{sec:zero_constr}
    In this subsection, we construct the value of a general zero-sum game under our integrability condition (Assumption~\ref{assum:bounded-functions}) and prove the existence of a global Markovian equilibrium. The main result, Theorem~\ref{thm:zerosum-construction}, establishes the existence of such an equilibrium and provides its explicit construction.

    For the middle value among the constants $a,b,c \in \mathbb{R}$ we use the notation 
	$\med(a,b,c) = \min\left\{\max\{a,b\}, \max\{a, c\}, \max\{b,c\}\right\}.$
    Using this notation, we define the function $\mathcal{W}: \mathbb{R}^4 \rightarrow \mathbb{R}$ via
\begin{equation}
    \label{eq:def_w}
	\mathcal{W}(f,g,h,w) = \begin{cases}
		h &\text{if } g \le h \le f \\
		\med (f,g,w) &\text{if } f \le h \le g \\
		w &\text{if } (h < f \wedge g \text{ or } h > f \vee g ) \text{ and } f \le w \le g  \\
		g &\text{if } h < f \wedge g \text{ and } w \geq g \\
		f &\text{if } h > f \vee g \text{ and } w \leq f \\
		\frac{gf-wh}{g-h+f-w} &\text{if } h \wedge w > f \vee g \text{ or } h \vee w < f \wedge g.
	\end{cases}
\end{equation}
(The role of $w$ will soon be clear). 
\begin{remark}\label{rem:monotonicity}
It can be seen that the function $\mathcal{W}$ is increasing in $w$, for fixed $f$, $g$ and $h$. Indeed, consider, e.g., the last case in \eqref{eq:def_w} and that the intersection between $p \mapsto (1-p)w+pg$ and $p \mapsto (1-p) f + p h$ is $p=\frac{f-w}{f-h+f-w}$ and use that the value for these functions at this intersection point is $\frac{gf-wh}{g-h+f-w}$, which is clearly increasing in $w$ (since one of the functions is increasing in $p$ and the other is decreasing in $p$, given the conditions in the last case in \eqref{eq:def_w}).
It can also be seen that the function $\mathcal{W}$ is bounded from below and above by 
$f\wedge g\wedge h$ and 
$f\vee g\vee h$, respectively. \defEnd
\end{remark}
Let us now define a function $V^0: E \rightarrow \mathbb{R}$ according to $x \mapsto f(x) \wedge g(x) \wedge h(x)$, and functions $V^{n+1}: E \rightarrow \mathbb{R}$, for $n \in \mathbb{N}_0$, according to
\[
	V^{n+1}(x) = \mathcal{W} (f(x),g(x),h(x), \alpha \Pi V^n(x)).
\] 
Using the choice of $V^0$ and Remark \ref{rem:monotonicity},
it can easily be shown that the sequence $(V^n(x))_{n \in \mathbb{N}}$ is increasing (for each fixed $x$). Using also that $\mathcal{W}$ is bounded in $w$, we find that the limit 
$V: E \rightarrow \mathbb{R}$ given by $V(x):= \lim_{n \rightarrow \infty} V^n(x)$ exists, is measurable and solves the equation
\begin{equation}
	\label{eq:value_fct_zero_sum}
	V(x)= \mathcal{W}(f(x),g(x),h(x), \alpha \Pi V(x)), \quad \text{for all } x \in E.
\end{equation}
Let us also  define the sets
\begin{align*}
B_1 & = \{x\in E: g(x) \le h(x) \le f(x) \}\\
B_2 & = \{x\in E: f(x) \le h(x) \le g(x), \alpha \Pi V(x) < f(x)\}\setminus B_1\\
B_3 & = \{x\in E: f(x) \le h(x) \le g(x), f(x) \le \alpha \Pi V(x) \le g(x) \}\setminus (\cup_{i=1}^2B_i )  \\
B_4 & = \{x\in E: f(x) \le h(x) \le g(x), g(x) < \alpha \Pi V(x) \}\setminus (\cup_{i=1}^3B_i)\\
B_5 & = \{x\in E: (h(x)< f(x) \wedge g(x) \text{ or } h(x)>f(x) \vee g(x))\text{ and } f(x) \le \alpha \Pi V(x) \le g(x) \}\setminus (\cup_{i=1}^4B_i)\\
B_6 & = \{x\in E: h(x) < f(x) \wedge g(x) \text{ and } \alpha \Pi V(x)  \ge g(x) \}\setminus (\cup_{i=1}^5B_i)\\
B_7 & = \{x\in E:h(x)> f(x) \vee g(x) \text{ and } \alpha \Pi V(x)  \le f(x)   \}\setminus (\cup_{i=1}^6B_i)\\
B_8 & = \{x\in E: h \wedge \alpha \Pi V(x) > f(x) \vee g(x) \text{ or } h(x) \vee \alpha \Pi V(x) < f(x) \wedge g(x) \}\setminus (\cup_{i=1}^7B_i).
\end{align*}
        
The following result describes how to construct an equilibrium from the function $V$ in \eqref{eq:value_fct_zero_sum} and the sets $B_i$.
\begin{theorem}
        \label{thm:zerosum-construction}

               Let $V: E \rightarrow \mathbb{R}$ be a measurable solution to \eqref{eq:value_fct_zero_sum} and choose $p^{(1)}, p^{(2)}: E \rightarrow [0,1]$ such that
		\begin{itemize}
			\item[(i)] $x\in B_1 \Rightarrow p^{(1)}(x)=p^{(2)}(x)=1$
			\item[(ii)] $ x\in B_2 \Rightarrow p^{(1)}(x)=1, p^{(2)}(x)=0$
			\item[(iii)] $x \in B_3 \Rightarrow p^{(1)}(x)=p^{(2)}(x)=0$
			\item[(iv)] $x\in B_4 \Rightarrow p^{(1)}(x)=0, p^{(2)}(x)=1$
			\item[(v)] $x\in B_5 \Rightarrow p^{(1)}(x)=p^{(2)}(x)=0$
			\item[(vi)] $ x\in B_6 \Rightarrow p^{(1)}(x)=0, p^{(2)}(x)=1$
			\item[(vii)] $x\in B_7\Rightarrow p^{(1)}(x)=1, p^{(2)}(x)=0$
			\item[(viii)] $x\in B_8 \Rightarrow p^{(1)}(x) = \frac{g(x)-\alpha \Pi V(x)}{f(x) - h(x)-\alpha \Pi V(x) + g(x)},  p^{(2)}(x) = \frac{f(x)-\alpha \Pi V(x)}{f(x) - h(x)-\alpha \Pi V(x) + g(x)}$. 
		\end{itemize} 
       (It is directly seen that such functions can be found.) Then $(p^{(1)},p^{(2)})$ is a global Markovian equilibrium with value $V$. In particular, every zero-sum game has a global Markovian equilibrium.
\end{theorem}

\begin{proof}
	It suffices to prove that $V$ and $(p^{(1)}, p^{(2)})$ satisfy \eqref{eq:Zero_All_Conditions} and that $V$ solves the integrability condition in Corollary \ref{cor:veri_zero_sum}(B)  (proved last). 	In fact we will prove that \eqref{eq:Zero_v_Maxi}, \eqref{eq:Zero_Maxi_Stop} and \eqref{eq:Zero_Maxi_Cont} hold, and note that the remaining parts of \eqref{eq:Zero_All_Conditions}, i.e., \eqref{eq:Zero_v_Mini}, \eqref{eq:Zero_Mini_Stop} and \eqref{eq:Zero_Mini_Cont}, can be similarly proved. It can be verified that the premises in (i)--(viii)  describe all possible cases and we will show that \eqref{eq:Zero_v_Maxi}, \eqref{eq:Zero_Maxi_Stop} and \eqref{eq:Zero_Maxi_Cont} hold for these cases separately.  
	
	(i): For $x \in B_1$, i.e., with $g(x) \le h(x) \le f(x)$, we will use that by assumption $p^{(1)}(x)=p^{(2)}(x)=1$. With this we have 
	\begin{align*}
		&\max \left\{(1-p^{(2)}(x)) \alpha \Pi V(x) + p^{(2)}(x) g(x), (1- p^{(2)}(x)) f(x) + p^{(2)}(x) h(x) \right\} \\
		& = \max \left\{ g(x),h(x) \right\}=h(x),
	\end{align*} which implies that \eqref{eq:Zero_Maxi_Stop} holds, while \eqref{eq:Zero_Maxi_Cont} is void. Moreover, we obtain $\max \left\{ g(x),h(x) \right\} = h(x)=V(x)$, since $V$ solves \eqref{eq:value_fct_zero_sum}. Hence, \eqref{eq:Zero_v_Maxi} holds.
	
	(ii): 
    We will use here that $p^{(1)}(x)=1$ and $p^{(2)}(x)=0$. With this we have
	\begin{align*}
		&\max \left\{(1-p^{(2)}(x)) \alpha \Pi V(x) + p^{(2)}(x) g(x), (1- p^{(2)}(x)) f(x) + p^{(2)}(x) h(x) \right\} \\
		& = \max \left\{ \alpha \Pi V(x),f(x) \right\}.
	\end{align*} By assumption $\alpha \Pi V(x) < f(x)$, which implies that  \eqref{eq:Zero_Maxi_Stop} holds, while is \eqref{eq:Zero_Maxi_Cont} void. Moreover, we obtain that $\max \left\{ \alpha \Pi V(x),f(x) \right\} = f(x)=V(x)$, since $V$ solves \eqref{eq:value_fct_zero_sum}. Hence, \eqref{eq:Zero_v_Maxi} holds.
	
	(iii): We will use here that $p^{(1)}(x)=p^{(2)}(x)=0$. With this we have
	\begin{align*}
		&\max \left\{(1-p^{(2)}(x)) \alpha \Pi V(x) + p^{(2)}(x) g(x), (1- p^{(2)}(x)) f(x) + p^{(2)}(x) h(x) \right\} \\
		& = \max \left\{ \alpha \Pi V(x),f(x) \right\}.
	\end{align*} By assumption $f(x) \le \alpha \Pi V(x)$, which implies that \eqref{eq:Zero_Maxi_Cont} holds, while \eqref{eq:Zero_Maxi_Stop} is void. Moreover, we obtain that $\max \left\{ \alpha \Pi V(x),f(x) \right\} = \alpha \Pi V (x)=V(x)$, since $V$ solves \eqref{eq:value_fct_zero_sum}. Hence, \eqref{eq:Zero_v_Maxi}  holds.
	
	(iv): We will use here $p^{(1)}(x)=0$ and $p^{(2)}(x)=1$. With this we have
	\begin{align*}
		&\max \left\{(1-p^{(2)}(x)) \alpha \Pi V(x) + p^{(2)}(x) g(x), (1- p^{(2)}(x)) f(x) + p^{(2)}(x) h(x) \right\} \\
		& = \max \left\{ g(x),h(x) \right\}.
	\end{align*} By assumption $g(x) \ge h(x)$, which implies that \eqref{eq:Zero_Maxi_Cont} holds, while \eqref{eq:Zero_Maxi_Stop} is void. Moreover, we obtain that $\max \left\{ g(x),h(x) \right\} = g(x)=V(x)$, since $V$ solves \eqref{eq:value_fct_zero_sum} and $g(x) < \alpha \Pi V(x)$. Hence, \eqref{eq:Zero_v_Maxi} holds.
	
	(v): We will use here $p^{(1)}(x)=p^{(2)}(x)=0$. With this we have
	\begin{align*}
		&\max \left\{(1-p^{(2)}(x)) \alpha \Pi V(x) + p^{(2)}(x) g(x), (1- p^{(2)}(x)) f(x) + p^{(2)}(x) h(x) \right\}\\
		& = \max \left\{ \alpha \Pi V(x),f(x) \right\}.
	\end{align*} By assumption $f(x) \le \alpha \Pi V(x)$, which implies that \eqref{eq:Zero_Maxi_Cont} holds, while \eqref{eq:Zero_Maxi_Stop} is  void. Moreover, we obtain that $\max \left\{ \alpha \Pi V(x),f(x) \right\} = \alpha \Pi V(x)=V(x)$, since $V$ solves \eqref{eq:value_fct_zero_sum}. 
 Hence, \eqref{eq:Zero_v_Maxi} holds.

	(vi): We will use here $p^{(1)}(x)=0$ and $p^{(2)}(x)=1$. With this we have
	\begin{align*}
		&\max \left\{(1-p^{(2)}(x)) \alpha \Pi V(x) + p^{(2)}(x) g(x), (1- p^{(2)}(x)) f(x) + p^{(2)}(x) h(x) \right\}\\
		& = \max \left\{ g(x),h(x) \right\}.
	\end{align*} By assumption $g(x) > h(x)$, which implies that \eqref{eq:Zero_Maxi_Cont} holds, while  \eqref{eq:Zero_Maxi_Stop} is void. Moreover, we obtain that $\max \left\{ g(x),h(x) \right\} = g(x)=V(x)$, since $V$ solves \eqref{eq:value_fct_zero_sum}.   Hence, \eqref{eq:Zero_v_Maxi} holds.
	
	(vii): We will use here $p^{(1)}(x)=1$ and $p^{(2)}(x)=0$. With this we have
	\begin{align*}
		&\max \left\{(1-p^{(2)}(x)) \alpha \Pi V(x) + p^{(2)}(x) g(x), (1- p^{(2)}(x)) f(x) + p^{(2)}(x) h(x) \right\}\\
		& = \max \left\{ \alpha \Pi V(x),f(x) \right\}.
	\end{align*} By assumption $f(x) \ge \alpha \Pi V(x)$, which implies that \eqref{eq:Zero_Maxi_Stop} holds, while \eqref{eq:Zero_Maxi_Cont} is void. Moreover, we obtain that $\max \left\{ \alpha \Pi V(x),f(x) \right\} = f(x)=V(x)$, since $V$ solves \eqref{eq:value_fct_zero_sum}.      Hence, \eqref{eq:Zero_v_Maxi} holds.
	
	(viii): Since $p^{(2)}(x) = \frac{f(x)-\alpha \Pi V(x)}{f(x) - h(x)-\alpha \Pi V(x) + g(x)}$ we have 
	\begin{align*}
		&\max \left\{(1-p^{(2)}(x)) \alpha \Pi V(x) + p^{(2)}(x) g(x), (1- p^{(2)}(x)) f(x) + p^{(2)}(x) h(x) \right\}\\
		& = \max \left\{ \frac{g(x)f(x)-\alpha \Pi V(x) h(x)}{f(x)-h(x)-\alpha \Pi V(x)+g(x)}, \frac{g(x)f(x)-\alpha \Pi V(x)h(x)}{f(x)-h(x)-\alpha \Pi V(x)+g(x)} \right\} \\
		& = \frac{g(x)f(x)-\alpha \Pi V(x)h(x)}{f(x)-h(x)-\alpha \Pi V(x)+g(x)}= V(x),
	\end{align*} since $V$ solves \eqref{eq:value_fct_zero_sum}, so that \eqref{eq:Zero_v_Maxi} holds. Moreover, the same computation, using $p^{(1)}(x) = \frac{g(x)-\alpha \Pi V(x)}{f(x) - h(x)-\alpha \Pi V(x) + g(x)}$, immediately implies that  \eqref{eq:Zero_Maxi_Cont} and \eqref{eq:Zero_Maxi_Stop} hold. 
	
	It is clear that 
	$|V(x)| \le |f(x)| \vee |g(x)| \vee |h(x)|$ for all $x \in E$, since  $V$ solves \eqref{eq:value_fct_zero_sum} and $\cal W$ is bounded from below by $f \wedge g \wedge h$ and above by $f \vee g \vee h$. 
	Hence, the required integrability in Corollary~\ref{cor:veri_zero_sum}(B) is satisfied (see Assumption~\ref{assum:bounded-functions} and subsequent observations). 
\end{proof}

	\begin{remark}
		From Remark~\ref{rem:value} and Theorem~\ref{thm:zerosum-construction} we can now conclude that there is a unique solution to \eqref{eq:value_fct_zero_sum}. Indeed, by Remark~\ref{rem:value}, the value of any zero-sum game is unique. Since by Theorem~\ref{thm:zerosum-construction}, every solution of \eqref{eq:value_fct_zero_sum} is the value of an equilibrium, the claim follows. \defEnd
	\end{remark}	

    \begin{remark}\label{comment-on-middle-val-assum} 
		Let us briefly relate our result to the existence results in the literature. As outlined in the introduction, characterization results have only been obtained under the assumption $F^1 \le H^1 \le G^1$ (e.g., \cite{ElbakidzeMarkovGame} for Markov games and \cite{OhtsuboTerminating} for general zero-sum games). Without this assumption only the existence of  $\epsilon$-equilibria (e.g., \cite{RosenbergSolanVieilleZeroSum}) could be established. Here, we describe the explicit \emph{construction} of an equilibrium under a mere integrability assumption. 
		\defEnd
	\end{remark}

\begin{remark}
    \label{rem:zero_sum_pure_equi}
    If 
    \begin{align}\label{assum:h-middle}
		h(x) = \med (f(x), h(x), g(x)), \enskip \text{for all } x \in E,
\end{align}
i.e., the function $h$ is in the middle, then we obtain from Theorem~\ref{thm:zerosum-construction} that a \emph{pure} global Markovian equilibrium exists. Indeed, since \eqref{assum:h-middle} holds, for every $x \in E$, it holds that 
$x \in \cup_{i=1}^4B_i$. Hence, $p^{(1)}$ and $p^{(2)}$ from Theorem~\ref{thm:zerosum-construction} are pure strategies. \defEnd
\end{remark}

	\subsection{Zero-sum Games where Randomization is Necessary} 
	\label{sec:zero_random}
	In the previous section (see Remark~\ref{rem:zero_sum_pure_equi}) we found that if the functions $f,g$ and $h$ satisfy the condition \eqref{assum:h-middle}, then a pure equilibrium always exits. A natural question to ask is if we should expect a pure equilibrium to exist also in case this condition is not satisfied. Note that if \eqref{assum:h-middle} does not hold then at least one of the sets
	\begin{align}\label{sets:no-pure}
		M_1:=\{x \in E: (f \vee g) (x) < h(x)\} \text { and } 
		M_2:=\{x \in E: h(x)< (f \wedge g) (x )\},
	\end{align}
	is non-empty. We here give an answer to this question by providing conditions under which no pure global Markovian equilibrium exists.  
	The first result, Theorem \ref{thm:suff-nec-cond-no-pure}, gives necessary and sufficient conditions for the non-existence of pure equilibria in terms of the value function of the game. The second result, Proposition~\ref{cor:suffcond-no-pure}, which relies on the first result, gives sufficient conditions for the non-existence of pure equilibria in terms of the payoff functions $f,g$ and $h$ and certain standard optimal stopping problems. Both results are, as we will see, obtained from the Wald-Bellman system equilibrium characterization in Corollary \ref{cor:veri_zero_sum}.

	\begin{theorem} [Sufficient and necessary conditions for non-existence of pure equilibria]\label{thm:suff-nec-cond-no-pure}
		Let $V$ denote the value function of the zero-sum game.
		Then, a pure global Markovian equilibrium exists if and only if
		\begin{align}\label{eq1:suff-nec-cond-no-pure}
			h(x) \vee  \alpha \Pi V(x) \geq  (f \wedge g) (x) \text{ and } h(x) \wedge  \alpha \Pi V(x) \leq  (f \vee g) (x), \text{ for all } x \in E. 
		\end{align}
	\end{theorem}

	\begin{remark}
		\label{rem:intuition_randomization}
		The condition~\eqref{eq1:suff-nec-cond-no-pure} has a clear interpretation. Indeed, at every time point the players can decide to continue or to stop the game. If both players choose the same action (i.e., both continue or both stop), then the reward of player 1 is $h$ (when both stop) or $\alpha \Pi V$ (when both continue). If one player stops and the other continues the reward of player $1$ is $f$ if player $1$ stops and $g$ if player $2$ stops. Hence, condition~\eqref{eq1:suff-nec-cond-no-pure} requires that the rewards associated to choosing the same action do not dominate or are not dominated by the reward associated to choosing different actions. We immediately see that condition~\eqref{eq1:suff-nec-cond-no-pure} should be necessary for the existence of a pure equilibrium: Suppose that \eqref{eq1:suff-nec-cond-no-pure} does not hold. Suppose first that the reward associated to choosing the same action dominates the reward associated to choosing different actions. Then player 1 would always want to choose the action player 2 chooses and player 2 would always want to choose the action that player 1 does not choose. Hence, no pure equilibrium can exist (as detailed in the proof below). An analogous argument works also for the case that the reward associated to choosing the same action is dominated by the reward associated to choosing different actions. \defEnd
	\end{remark}
	
	\begin{proof}[Proof of Theorem~\ref{thm:suff-nec-cond-no-pure}]
		Let us first show that if there exists a pure global Markovian equilibrium $(p^{(1)},p^{(2)})$, then \eqref{eq1:suff-nec-cond-no-pure} holds. By Corollary \ref{cor:veri_zero_sum} we know that $V$ and $(p^{(1)},p^{(2)})$ satisfy 
		\eqref{eq:Zero_v_Maxi}--\eqref{eq:Zero_v_Mini} for each $x\in E$. For a pure equilibrium we have four possible cases for $(p^{(1)},p^{(2)})$ for each $x\in E$, and given each of these it is easily verified, that \eqref{eq1:suff-nec-cond-no-pure} holds:
		
		Case 1: $(p^{(1)}(x),p^{(2)}(x))=(1,1)$. In this case both players stop immediately and the equilibrium value therefore satisfies $V(x)=h(x)$ (by definition, see \eqref{eq:zero_value_fct}). Moreover, in this case it follows from \eqref{eq:Zero_v_Maxi}--\eqref{eq:Zero_v_Mini} that 
		$h(x) = \max \left\{g(x),  h(x) \right\}$ and $h(x) = \min \left\{f(x),  h(x)\right\}$.
		It follows that, $g(x)\leq h(x) \leq f(x)$, which in turn directly implies that \eqref{eq1:suff-nec-cond-no-pure} holds.

		Case 2: $(p^{(1)}(x),p^{(2)}(x))=(1,0)$. In this case $V(x)=f(x)$ and \eqref{eq:Zero_v_Maxi}--\eqref{eq:Zero_v_Mini} imply that 
		$f(x) = \max \left\{\alpha \Pi V(x),  f (x) \right\}$ and 
		$f(x) = \min \left\{f(x),  h(x)\right\}$.
		Hence, $\alpha \Pi V(x)\leq f(x) \leq h(x)$, which in turn implies that \eqref{eq1:suff-nec-cond-no-pure} holds.

		Case 3: $(p^{(1)}(x),p^{(2)}(x))=(0,1)$. In this case $V(x)=g(x)$ and 
		\eqref{eq:Zero_v_Maxi}--\eqref{eq:Zero_v_Mini} imply that 
		$g(x) = \max \left\{g(x),  h(x) \right\}$ and $g(x) = \min \left\{\alpha \Pi V(x),  g(x)\right\}$. 
		Hence, $h(x) \leq g(x) \leq \alpha \Pi V(x) $, which in turn implies that \eqref{eq1:suff-nec-cond-no-pure} holds.  
		
		Case 4: $(p^{(1)}(x),p^{(2)}(x))=(0,0)$. In this case $V(x)=\alpha \Pi V(x)$ and 
		\eqref{eq:Zero_v_Maxi}--\eqref{eq:Zero_v_Mini} imply that 
		$\alpha \Pi V(x) = \max \left\{\alpha \Pi V(x),  f(x) \right\}$ 
		and 
		$\alpha \Pi V(x) = \min \left\{\alpha \Pi V(x),  g(x)\right\}$. 
		Hence, $f(x) \leq \alpha \Pi V(x) \leq  g(x)$, which in turn implies that \eqref{eq1:suff-nec-cond-no-pure} holds.

		Let us now prove the opposite implication. If \eqref{eq1:suff-nec-cond-no-pure} holds, then $x \notin B_8$
        and hence the strategies $p^{(1)}$ and $p^{(2)}$ in Theorem~\ref{thm:zerosum-construction} are pure, which directly proves the claim.
                	\end{proof}

	\begin{prop}[Sufficient conditions for non-existence of pure equilibria] \label{cor:suffcond-no-pure}
		Consider the sets $M_i,i=1,2$ defined in \eqref{sets:no-pure} and the standard stopping problems
		\begin{align*}
			V_{M_1}(x)&:= \sup_{\tau \in \mathcal{T}_1, \tau \leq \tau_{M_1}} \mathbb{E}_x \left[ \alpha^\tau k_1(X_\tau) \right], \quad x \in M_1,\\
			V_{M_2}(x)&:= \inf_{\tau \in \mathcal{T}_2, \tau \leq \tau_{M_2}} \mathbb{E}_x \left[ \alpha^\tau k_2(X_\tau) \right], \quad x \in M_2,
		\end{align*} 
		where $ \tau_{M_i}:= \inf \{{t \geq0 : X_t \notin M_i}\}$ and 
		\begin{align}
			k_1(x)& :=   \mathbb{I}_{\{x \in M_1\}} f(x) + \mathbb{I}_{\{x \notin M_1\}} ( f \wedge h)(x),\label{cor:suffcond-no-pure:xyz}\\
			k_2(x)& :=   \mathbb{I}_{\{x \in M_2\}} g(x) + \mathbb{I}_{\{x \notin M_2\}} ( g \vee h)(x) \nonumber.
		\end{align}
		(A) If $V_{M_1}(x_0)>(f \vee g)(x_0)$ for some $x_0 \in M_1$ then no pure global Markovian equilibrium exists. (B) If $V_{M_2}(x_0)<(f \wedge g)(x_0)$ for some $x_0 \in M_2$ then no pure global Markovian equilibrium exists.
	\end{prop}

	\begin{proof}	 
		(A) 
		We prove the statement using a contradiction argument. Suppose $(p^{(1)},p^{(2)})$ is a pure global Markovian equilibrium with value function $V$. 
		
		Recall that $x\in M_1$ means that $(f \vee g) (x) < h(x)$. Note that $p^{(1)}(x)=p^{(2)}(x)=1$ cannot hold since in this case $V(x)=h(x)$, which means that player $2$ (the minimizer) would deviate and obtain $f(x)<h(x)=V(x)$ (in other words, \eqref{eq:Zero_v_Mini} implies the contradiction 
		$V(x) = \min\left\{f(x),h(x)\right\} = f(x)<h(x)=V(x)$). 
		Similarly, $p^{(1)}(x)=0, p^{(2)}(x)=1$ cannot hold since in this case $V(x)=g(x)$, which is a contradiction to 
		\eqref{eq:Zero_v_Maxi} which implies that $V(x) = \max \left\{g(x),h(x)\right\} = h(x)>g(x)=V(x)$. 
		We conclude that $p^{(2)}(x)=0$, for $x\in M_1$, i.e., player $2$ does not stop on $M_1$.  
		
		Since player $2$ does not stop on $M_1$, it follows that the equilibrium 
		value $V$ must be dominated by the value of the one-player optimization problem where player $1$ obtains $f(x)$ if stopping on $M_1$ 
		and $(f \wedge h)(x)$ if stopping directly after leaving $M_1$ (cf. the function $k_1$ in \eqref{cor:suffcond-no-pure:xyz}), specifically for $x=x_0$
		\begin{align*}
			V(x_0) \geq V_{M_1}(x_0) > (f \vee g)(x_0)
		\end{align*}
		(the second inequality is a condition in the statement of the result). This implies that 
		$V(x_0)>f(x_0)$ and using also the finding $p^{(2)}(x_0)=0$, we conclude with \eqref{eq:Zero_v_Maxi} that 
		$V(x_0) = \max \left\{\alpha \Pi V(x_0),  f(x_0) \right\} = \alpha \Pi V(x_0)$. Hence, 
		\begin{align*} 
			\alpha \Pi V(x_0) > (f \vee g)(x_0). 
		\end{align*}
		Using this and that $x_0 \in M_1$ we find that $h (x_0) \wedge  \alpha \Pi V(x_0) > (f \vee g)(x_0)$, which means that \eqref{eq1:suff-nec-cond-no-pure} does not hold. This is the desired contradiction (cf. Theorem~\ref{thm:suff-nec-cond-no-pure}).
		
		(B): The proof is analogous to the one above.
	\end{proof}

	\subsection{A Zero-sum Game With a Randomized But No Pure Equilibrium}\label{sec:zerosum-randomNE-no-mix-example}
	Here we present a game with only two states with a global Markovian \textit{randomized} equilibrium, which does not have a global Markovian \textit{pure} equilibrium. 
	The construction relies on the idea presented in Remark~\ref{rem:intuition_randomization}. Namely, if the game is in state $1$, then one player prefers either that both players do not stop or that they stop simultaneously, while the other player prefers that exactly one player stops. 
	As explained before, with these preferences, it indeed seems reasonable that only randomized equilibria are possible. 
	
	Let $E=\{1,2\}$, $\alpha = 4/5$, $f(1)=g(1)=0$, $h(1)=2$, $f(2)=5$, $g(2)=3$, $h(2)=4$ and
	\[
	\Pi = \begin{pmatrix}
		1/2 & 1/2 \\ 0 & 1 
	\end{pmatrix}.
	\] 
	We first note that this example satisfies the conditions of Proposition~\ref{cor:suffcond-no-pure} with $M_1= \{1\}.$ Hence, by Proposition~\ref{cor:suffcond-no-pure} no global Markovian pure equilibrium exists. Nonetheless a unique global Markovian randomized equilibrium exists:
	
	\textit{Claim:} 
	The randomized stopping strategy pair 
	$\left(p^{(1)}, p^{(2)}\right)$ where
	\begin{equation}
		\label{eq:zero_example_mixed_equi}
		p^{(1)}(1) = p^{(2)}(1)=1/2 \quad \text{and} \quad p^{(1)}(2)=p^{(2)}(2)=1
	\end{equation}
	is the unique global Markovian equilibrium.  
	
	\begin{proof}
		 We start by showing that	$\left(p^{(1)},p^{(2)}\right)$ given by \eqref{eq:zero_example_mixed_equi} is indeed an equilibrium. For this it suffices that $p^{(1)}$, $p^{(2)}$ and $V$ given by $V(1)=1$ and $V(2)=4$ satisfy \eqref{eq:Zero_All_Conditions}. This immediately follows from
		\begin{align*}
			\tfrac{4}{5} (1-p^{(2)}(2)) V(2) + 3p^{(2)}(2) &= 3 < 4 = 5(1-p^{(2)}(2)) + 4p^{(2)}(2) \\
			\tfrac{4}{5} (1-p^{(1)}(2)) V(2) + 5p^{(1)}(2) &= 5 > 4 = 3(1-p^{(1)}(2)) + 4p^{(1)}(2) \\
			(1-p^{(2)}(1)) \left[ \tfrac{2}{5} V(1) + \tfrac{8}{5} \right] &= 1 = 2p^{(2)}(1) \\
			(1-p^{(1)}(1)) \left[ \tfrac{2}{5} V(1) + \tfrac{8}{5} \right] &= 1 = 2p^{(1)}(1).
		\end{align*}
		
	To prove uniqueness, suppose $\left(p^{(1)}, p^{(2)}\right)$ to be any global Markovian equilibrium. From Remark \ref{rem:value}, we know that it has value $V(1)=1$ and $V(2)=4$ as well. By Corollary~\ref{cor:veri_zero_sum}, $p^{(1)}$, $p^{(2)}$ and $V$ satisfy \eqref{eq:Zero_All_Conditions}. Using that we know $V$, \eqref{eq:Zero_v_Maxi} for $x=2$ reads
			\[
			4= \max \left\{ \tfrac{4}{5}(1-p^{(2)}(2)) 4 + 3 p^{(2)}(2), 5(1-p^{(2)}(2)) + 4 p^{(2)}(2) \right\}.
			\]
	which implies $p^{(2)}(2)=1$ and similarly we find $p^{(1)}(2)=1$.
	 \eqref{eq:Zero_v_Maxi} for the state $x=1$ simplifies to
	\[
	1= \max \left\{ 2(1-p^{(2)}(1)) , 2p^{(2)}(1) \right\},
	\]
	yielding $p^{(2)}(1)=\frac{1}{2}$. 
		With analogous arguments it can also be shown that $p^{(1)}(1)=\tfrac{1}{2}$.
\end{proof}

	\section{Results for Symmetric Games}
	\label{sec:symmetric}
We have a symmetric game when the players' payoff functions coincide, i.e., when
	\begin{align}\label{eq:symmetric-game-def}
		f_1=f_2=f, \quad g_1=g_2=g \quad \text{and} \quad h_1=h_2=h.
	\end{align}
	In the case of a countable state space it holds that every symmetric game admits a global Markovian \textit{symmetric} equilibrium -- by which we mean a global Markovian equilibrium $\left(p^{(1)},p^{(2)}\right)$ satisfying $p^{(1)}=p^{(2)}$; see Section~\ref{sec:countable-symNE} below. Here, however, we consider general state space and the case $f=h$ -- i.e., the payoffs of stopping first and stopping simultaneously coincide -- and search for 
	global Markovian symmetric equilibria. 	
	Indeed under this condition we obtain that such an equilibrium can always be constructed explicitly in terms of an associated standard optimal stopping problem:

	\begin{theorem}[Existence and construction of a symmetric equilibrium] 
		\label{thm:equilibrium_symmetric}
		Suppose $f=h$. Let $V: E\rightarrow \mathbb{R}$ be given by
		\[
		V(x):= \sup_{\tau \in \mathcal{T}_1} \mathbb{E}_{x} [ \alpha^\tau f(X_\tau)], \enskip x \in E.
		\] 
		Choose a measurable function $p: E \rightarrow [0,1]$ so that
		\begin{itemize}
			\item[(i)] $\alpha \Pi V(x) = f(x) , g(x) = f(x) \Rightarrow p(x) \in [0,1]$\vspace{-2mm}
			\item[(ii)] $\alpha \Pi V(x) = f(x), g(x) \neq f(x) \Rightarrow p(x) = 0$\vspace{-2mm}
			\item[(iii)] $\alpha \Pi V(x) < f(x), g(x) > f(x) \Rightarrow p(x) = \frac{ f(x) - \alpha \Pi V(x) }{g(x) -\alpha \Pi V(x)}$\vspace{-2mm}
			\item[(iv)] $\alpha \Pi V(x) < f(x), g(x) \le f(x) \Rightarrow p(x)=1$\vspace{-2mm}
			\item[(v)] $\alpha \Pi V(x)> f(x) \Rightarrow p(x)=0$.\vspace{-2mm}
		\end{itemize}
		(It is directly seen that such a function can be found). 
		Then $\left(p^{(1)},p^{(2)}\right)$ with $p^{(1)}=p^{(2)}=p$ is a (symmetric) global Markovian equilibrium with values $V^{(1)}=V^{(2)}=V$. 
	\end{theorem}
	The proof of this result uses Theorem~\ref{thm:veri} after verifying that $V$ and $p$ satisfy \eqref{eq:verification_all} with $V^{(i)}=V,p^{(i)}=p,i=1,2$ and \eqref{eq:veri_int}. Indeed, the proof is similar to that of Theorem~\ref{thm:zerosum-construction} and is for the sake of brevity not included in the paper. Let us highlight that the probability with which the players randomize in case (iii) is exactly the probability that makes the other player indifferent between stopping and continuing, i.e., for which
	\[
	(1-p(x)) \alpha \Pi V(x) + p(x) g(x) = (1-p(x)) f(x) + p(x) h(x).
	\]
	An interesting special case is when $f$ is strictly $\alpha$-excessive, i.e., $\alpha \Pi f < f$, and $g>f$. Locally, the players then want to stop as early as possible on the one hand because of the excessivity, but on the other hand prefer to stop later than the other player because of $g>f$. This situation is known as the war of attrition and leads to a situation where randomization seems rational. Indeed:
	
	\begin{corollary}
		Suppose $f$ is strictly $\alpha$-excessive and $g>f=h$. Then, for all $x\in E$
		\[p(x) = \frac{f(x) -\alpha \Pi f(x)}{g(x) - \alpha \Pi f(x)}\in(0,1)\]
		and $\left(p^{(1)},p^{(2)}\right)$ with $p^{(1)}=p^{(2)}=p$ is a (symmetric) global Markovian equilibrium with values $V^{(1)}=V^{(2)}=f$.
	\end{corollary}

	\begin{proof}
		For $\alpha$-excessive rewards, we have 
		$V(x)=\sup_{\tau \in \mathcal{T}_1} \mathbb{E}_{x} [ \alpha^\tau f(X_\tau)]=f(x)$, since immediate stopping is optimal. 
		We are thus in case (iii) of Theorem~\ref{thm:equilibrium_symmetric}, which proves the result (the conclusion $p(x) \in(0,1)$ follows from the assumptions for $f$ and $g$). 
	\end{proof}
	
	Note that although $V^{(1)}=V^{(2)}=f=h$ in the result above, it holds that immediate stopping (for any, or both, of the players) does not constitute an equilibrium. 
	
	\section{General Existence of Equilibria for Countable State Spaces}\label{sec:Existence}
	In this section we establish the following general existence result. 
	
	\begin{theorem}
		\label{thm:Existence}
		Let the state space $E$ be countable. Then a global Markovian randomized equilibrium $\left(p^{(1)},p^{(2)}\right)$ exists. 
	\end{theorem}
	Based on the best response mapping \eqref{best-response-mapping} we define a \emph{two-player best response mapping} as follows. As usual we equipp the countable set $E$ with the discrete $\sigma$-algebra. Hence, every function $p:E \rightarrow [0,1]$ is measurable, that is  $\mathcal{M}(E,[0,1])=[0,1]^E$.

	\begin{definition} [Two-player best response mapping] We call the set-valued mapping 
		\begin{align*}
			F:[0,1]^E \times [0,1]^E &\rightarrow {\cal P}\left([0,1]^E\right) \times {\cal P}\left([0,1]^E\right) \\
			 \left(p^{(1)},p^{(2)}\right) &\mapsto F\left(p^{(1)}, p^{(2)}\right):= \text{BR}^{(1)}\left(p^{(2)}\right) \times \text{BR}^{(2)}\left(p^{(1)}\right),
		\end{align*}
		the two-player best response mapping.
		\defEnd
\end{definition}
The interpretation is that $F$ takes as input a pair of strategies $\left(p^{(1)},p^{(2)}\right)$ and outputs 
	all stopping strategy pairs $\left({\tilde{p}}^{(1)}, {\tilde{p}}^{(2)}\right)$ such that 
	${\tilde{p}}^{(1)}$ is a best response to $p^{(2)}$ and
	${\tilde{p}}^{(2)}$ is a best response to $p^{(1)}$. Thus, it is immediately seen that a fixed point in the two-player best response mapping
	$\left(p^{(1)}, p^{(2)}\right) \in F\left(p^{(1)},p^{(2)}\right)$ 
	is a global Markovian equilibrium. Indeed, by definition of $F$, it holds that if $\left(p^{(1)}, p^{(2)}\right)$ is a fixed point then $p^{(2)} \in \text{BR}^{(2)}\left(p^{(1)}\right)$ and $p^{(1)} \in \text{BR}^{(1)}\left(p^{(2)}\right)$, which means that $\left(p^{(1)}, p^{(2)}\right)$ is a global
	Markovian randomized equilibrium (for details in this argument compare also Definition~\ref{def:NE} and Proposition~\ref{thm:characterization_optimal_strategies}). Hence, in order to prove equilibrium existence, it suffices to show that the set-valued mapping $F$ has a fixed point.

	\begin{remark}\label{rem:existPF}
		Our approach to proving that $F$ has a fixed point relies on a version of Kakutani's fixed point theorem (see \cite[Theorem 7.8.6]{DugundjiFixedPoint2003}). 
		It states that any set-valued mapping from a compact and convex subset of a locally convex space to its power set that is upper-semicontinuous and has non-empty, closed and convex values has a fixed point. 
		A central difficulty in applying this result for any set-valued mapping $G:C \rightarrow {\cal P}\left(C\right)$ is to find a topological space $X$ such that on the one hand $G$ is upper-semicontinuous with respect to the topology and on the other hand $C \subseteq X$ is a compact (and convex) subset. 
		In our case, we essentially  need to find a topology such that the best response mapping $F$ is upper semicontinuous and such that the set in which all best responses lie in is compact. 
		We  equip the set of all strategies $[0,1]^E$ with the topology of pointwise convergence, i.e., the product topology. We remark that we require $E$ to be countable in order to conclude that $[0,1]^E$ is metrizable, which in turn is the reason we can use the sequential criterion \eqref{eq:claim_BR_usc} to prove equilibrium existence (cf. the proof below).
		\defEnd 
	\end{remark}

	\begin{proof}[Proof of Theorem~\ref{thm:Existence}]
		As argued above, it suffices to prove that the two-player best response mapping $F$ has a fixed point.  
		In particular, it suffices to show that $F$ and the space $[0,1]^E \times [0,1]^E$ equipped with the product topology satisfy the assumptions of the version of Kakutani's fixed point theorem stated in Remark~\ref{rem:existPF}.

		It is immediate that the space $[0,1]^E \times [0,1]^E$ is convex. Moreover, by Tychonoff's theorem, this space is, as a product of compact spaces, compact. 
		Finally, by \cite[§18.3]{KoetheTopologoicalVectorSpaces}, we note that this space equipped with the product topology is locally convex. Hence, $[0,1]^E \times [0,1]^E$ satisfies the assumptions of Kakutani's fixed point theorem. 
		
		It follows from Proposition~\ref{thm:CharacterizationValues} that the values of $F$ are non-empty, closed and convex.
		Hence, it suffices to show that $F$ is upper semicontinuous. 
		It is however clear that it suffices to prove this component-wise, i.e., it suffices to prove that $\text{BR}^{(1)}$ is upper semicontinuous, which we will now do.
		Since $[0,1]^E$ is compact and Hausdorff, it suffices, by the closed graph theorem (\cite[Theorem 17.11]{AliprantisClosedGraphThm}), to show that the mapping $\text{BR}^{(1)}$ has a closed graph, 
		i.e., it suffices to show that the set $\{\left(p^{(1)}, p^{(2)}\right) : p^{(1)} \in \text{BR}^{(1)}\left(p^{(2)}\right)\}$ is closed. 
		However, since $E$ is countable it holds the space $[0,1]^E$ is metrizable, and it hence suffices to show that 
		\begin{equation}
			\label{eq:claim_BR_usc}
			p^{(i)}_n \rightarrow p^{(i)}, i=1,2 \text{ and } p_n^{(1)} \in \text{BR}^{(1)}\left(p^{(2)}_n\right) \text{ for all }n \in \mathbb{N}  
			\Rightarrow p^{(1)} \in \text{BR}^{(1)}\left(p^{(2)}\right),
		\end{equation}
		which we will now do. In order to highlight the dependence on the strategy of player $2$, i.e., ${p^{(2)}}$ we will here write 
		$\hat{\Pi}^{p^{(2)}}$ for the Markov kernel \eqref{markov-kernel}, 
		$\hat{r}^{p^{(2)}}$ for reward function \eqref{r-reward}, 
		$\hat{X}^{p^{(2)}}$ for the process given by Definition~\ref{def:X-hat}, 
		and $\hat{V}^{p^{(2)}}$ for the value function of the stopping problem \eqref{sec3:optstop2}.
		We also write, cf. \eqref{sec3:optstop2},  
		\[
		J^{p^{(2)}}_\tau (\hat{x}) = \mathbb{E}_{\hat{x}} \left[ \hat{r}^{p^{(2)}}(\hat{X}^{p^{(2)}}_\tau) \right]. 
		\] 
		Since $p_n^{(2)} \rightarrow p^{(2)}$ if and only if $p_n^{(2)}(x) \rightarrow p^{(2)}(x)$ for all $x \in E$, it is immediately clear that $\hat{r}^{p_n^{(2)}}(\hat{x}) \rightarrow \hat{r}^{p^{(2)}}(\hat{x})$ for all $\hat{x} \in E$.
		The proof of \eqref{eq:claim_BR_usc} consists of three steps:
		\begin{itemize}
			\item[(i)] We prove that for any sequence $p^{(2)}_n \rightarrow p^{(2)}$, we have that $\hat{V}^{p_n^{(2)}}(\hat{x}) \rightarrow \hat{V}^{p^{(2)}}(\hat{x})$ for all $\hat{x} \in \hat{E}$. \vspace{-2mm}
			\item[(ii)] We prove that for any sequence $p^{(2)}_n \rightarrow p^{(2)}$, we have that $\hat{\Pi}^{p_n^{(2)}}\hat{V}^{p_n^{(2)}}(\hat{x}) \rightarrow \hat{\Pi}^{p^{(2)}} \hat{V}^{p^{(2)}}(\hat{x})$ for all $\hat{x} \in \hat{E}$.\vspace{-2mm}
			\item[(iii)] We use (i) and (ii) to prove \eqref{eq:claim_BR_usc}.\vspace{-2mm}
		\end{itemize}
		
		(i): Since the states $\hat{x} \in \hat E_{S \cup K}$ are absorbing and the rewards for these states are independent of $p_n^{(2)}$ and $p^{(2)}$, the claim holds for $\hat{x} \in \hat E_{S \cup K}$. 
		Hence, we now consider the states $(x,C) = \hat{x} \notin \hat{E}_{S \cup K}$. Let us write $\tilde{\mathcal{T}}_1$ for the set of stopping times with respect to the filtration $\sigma(X_0, \ldots, X_n)$, $n \in \mathbb{N}_0$. 
		In Proposition~\ref{thm:characterization_optimal_strategies} we proved that for any $p^{(2)} \in [0,1]^E$ at least one strategy from $\tilde{\mathcal{T}}_1$ is optimal for \eqref{sec3:optstop2}, and hence
		\[
		\hat V^{p^{(2)}}(\hat{x}) = \sup_{\tau \in \tilde{\mathcal{T}}_1} J_\tau^{p^{(2)}}(\hat{x}). 
		\]
		Using the equality \eqref{eq:RelationRewardGameStopping} we find, for all $m \in \mathbb{N}_0$,
		\begin{align*}
			&\mathbb{E}_{(x,C)} \left[ \hat{r}^{p^{(2)}}(\hat{X}^{p^{(2)}}_m) | X_0, \ldots, X_m \right] \\
			&= \sum_{i=1}^m \left(\prod_{j=0}^{i-2} (1-p^{(2)}(X_j)) \right) p^{(2)}(X_{i-1}) \alpha^{i-1} g_1(X_{i-1})+ \left(\prod_{j=0}^m (1-p^{(2)}(X_j))\right) \alpha^m f_1(X_m)  \\
			&\quad+ \left(\prod_{j=0}^{m-1} (1-p^{(2)}(X_j)) \right) p^{(2)}(X_m) \alpha^m h_1(X_m).
		\end{align*} 
		Hence, for any stopping time $\tau \in \tilde{\mathcal{T}}_1$ with $\tau \le m$ we find measurable functions 
		$G_i: E^{i+1} \rightarrow \mathbb{R}$ ($i=0, \ldots, m$) 
		and $H: E^{m+1} \rightarrow \mathbb{R}$ 
		with 
		\begin{align*}
			G_i(X_0, \ldots, X_i) \le M,  \quad i = 0, \ldots, m, \quad \text{and} \quad  H(X_0, \ldots, X_m) \le M,
		\end{align*}
		such that 
		\begin{align*}
			&\mathbb{E}_{(x,C)}[\hat{r}^{p^{(2)}} (\hat{X}^{p^{(2)}}_\tau) | X_0, \ldots, X_m] \\
			&= \sum_{i=0}^m \prod_{j=0}^{i-1} (1-p^{(2)}(X_j)) p^{(2)}(X_i) G_i (X_0, \ldots, X_i) + \prod_{j=0}^m (1-p^{(2)}(X_j)) H(X_0, \ldots, X_m).
		\end{align*}
		(Recall that $M$ is defined in connection to Assumption~\ref{assum:bounded-functions}).
		Taking expectations we find
		\begin{align*}
			&\left| \mathbb{E}_{(x,C)} \left[ \hat{r}^{p^{(2)}}(\hat{X}_\tau^{p^{(2)}}) \right] - \mathbb{E}_{(x,C)} \left[\hat{r}^{p_n^{(2)}} (\hat{X}^{p_n^{(2)}}_\tau) \right] \right| \\
			&= \left| \mathbb{E}_{(x,C)} \left[ \sum_{i=0}^m \left( \prod_{j=0}^{i-1} (1-p^{(2)}(X_j)) p^{(2)}(X_i) - \prod_{j=0}^{i-1} (1-p_n^{(2)}(X_j)) p_n^{(2)}(X_i)  \right) G_i(X_0, \ldots, X_i) \right. \right. \\
			&\quad +\left. \left. \left( \prod_{j=0}^m (1-p^{(2)}(X_j)) - \prod_{j=0}^m (1-p_n^{(2)}(X_j)) \right) H(X_0, \ldots, X_m) \right] \right|.
		\end{align*}
		It follows that 
		\begin{align*}
			&\left| \sup_{\tau \le m} J_\tau^{p^{(2)}} (\hat{x}) - \sup_{\tau \le m} J_\tau^{p^{(2)}_n}  (\hat{x})\right| 
			\le \sup_{\tau \le m} \left| J_\tau^{p^{(2)}} (\hat{x}) -  J_\tau^{p^{(2)}_n}  (\hat{x})\right| \le  \\
			&\sup_{G_i, H \le M} \left| \mathbb{E}_{(x,C)} \left[ \sum_{i=0}^m \left( \prod_{j=0}^{i-1} (1-p^{(2)}(X_j)) p^{(2)}(X_i) - \prod_{j=0}^{i-1} (1-p_n^{(2)}(X_j)) p_n^{(2)}(X_i)  \right) G_i(X_0, \ldots, X_i) \right. \right. \\
			&\quad +\left. \left. \left( \prod_{j=0}^m (1-p^{(2)}(X_j)) - \prod_{j=0}^m (1-p_n^{(2)}(X_j)) \right) H(X_0, \ldots, X_m) \right] \right|.
		\end{align*}
		The random variable inside the last expectation is dominated by $2(m+1)M$ and converges pointwise to $0$, since $p_n^{(2)} \rightarrow p^{(2)}$ (by assumption). 
		Relying on dominated convergence we thus obtain
		\[
		\sup_{\tau \le m} J_\tau^{p_n^{(2)}} (\hat{x}) \rightarrow \sup_{\tau \le m} J_\tau^{p^{(2)}} (\hat{x}).
		\]
		Since the reward associated to stopping at a time later than $m$ is bounded by $\alpha^m M$ we obtain using dominated convergence that 
		\[
		\sup_{\tau \le m} J_{\tau}^{p^{(2)}}(\hat{x}) \overset{m \rightarrow \infty}{\rightarrow} \sup_{\tau \in \tilde{\mathcal{T}}_1} J_\tau^{p^{(2)}}(\hat{x}) = \hat{V}^{p^{(2)}}(\hat{x})
		\]
		and similarly for $p^{(2)}_n$. We conclude that the desired claim holds, i.e., 
		\[
		\hat{V}^{p_n^{(2)}}(\hat{x})=\sup_{\tau \in \tilde{\mathcal{T}}_1} J_\tau^{p^{(2)}_n} (\hat{x}) \rightarrow  \hat{V}^{p^{(2)}}(\hat{x}).
		\]
		(ii): Since the states $\hat{x} \in \hat{E}_{S \cup K}$ are absorbing, the claim holds (similarly to the above) for $\hat{x} \in \hat{E}_{S \cup K}$. 
		Hence, we now consider the states $(x,C) = \hat{x} \notin \hat{E}_{S \cup K}$. The claim then follows, by dominated convergence, according to 
		\begin{align*}
			&\hat{\Pi}^{p_n^{(2)}} V^{p_n^{(2)}}(x,C) \\
			&=  (1-p_n^{(2)}(x)) \int_E \alpha V^{p_n^{(2)}}(y)  \Pi(x, \text{d} y) +  {p_n^{(2)}}(x) \hat{V}^{{p_n^{(2)}}}(x,S) + (1-\alpha) (1-p_n^{(2)}(x)) \cdot 0 \\
			&\rightarrow (1-p^{(2)}(x)) \int_E \alpha  V^{p^{(2)}}(y) \Pi(x, \text{d} y) +  {p^{(2)}}(x) \hat{V}^{p^{(2)}}(x,S) + (1-\alpha) (1-p^{(2)}(x)) \cdot 0 \\
			&= \hat{\Pi}^{p^{(2)}} V^{p^{(2)}}(x,C).
		\end{align*}
		(iii): 
		In order to show that \eqref{eq:claim_BR_usc} holds, it suffices to show, by Proposition~\ref{thm:characterization_optimal_strategies}, that for any $x \in E$, it holds that
		\begin{itemize}
			\item[(a)] $p^{(1)}(x) = 0 \Rightarrow (x,C) \notin \hat D$\vspace{-2mm}
			\item[(b)] $p^{(1)}(x) \in (0,1) \Rightarrow (x,C) \in \hat I$\vspace{-2mm}
			\item[(c)] $p^{(1)}(x) =1 \Rightarrow (x,C) \in \hat I \cup \hat D$.
		\end{itemize}
		Recall that for $\hat{x}=(x,C)$ and $p_n^{(2)} \rightarrow p^{(2)}$ we have
		\begin{equation}
			\label{eq:usc_prelim}
			\hat{V}^{p_n^{(2)}}(\hat{x}) \rightarrow \hat{V}^{p^{(2)}}(\hat{x}), \quad \hat{\Pi}^{p_n^{(2)}} \hat{V}^{p_n^{(2)}}(\hat{x})  \rightarrow \hat{\Pi}^{p^{(2)}} \hat{V}^{p^{(2)}}(\hat{x}) \quad \text{and} \quad \hat{r}^{p^{(2)}_n}(\hat{x}) \rightarrow \hat{r}^{p^{(2)}}(\hat{x}).
		\end{equation}
		Below we show that (a)--(c) hold.  We use the definitions of $\hat D$ and $\hat I$ in \eqref{eq:indiff-cont-sets}  repeatedly. 
		
		(a): Suppose $p^{(1)}(x) = 0$. It suffices to show that $\hat{\Pi}^{p^{(2)}}\hat{V}^{p^{(2)}}(x,C) \ge \hat{r}^{p^{(2)}}(x,C)$. Since $p_n^{(1)}(x) \rightarrow p^{(1)}(x)$ we have that there is an $n_0 \in \mathbb{N}$ such that for all $n \ge n_0$ we have $p_n^{(1)}(x) < \frac{1}{2}$, which means, by Proposition~\ref{thm:characterization_optimal_strategies} that 
		$\hat{\Pi}^{p_n^{(2)}}\hat{V}^{p_n^{(2)}}(x,C) \ge \hat{r}^{p_n^{(2)}}(x,C)$ for all $n \ge n_0$, and using \eqref{eq:usc_prelim} we obtain the desired statement.
		
		(b): Suppose $p^{(1)}(x) \in (0,1)$. It suffices to show that $\hat{\Pi}^{p^{(2)}}\hat{V}^{p^{(2)}}(x,C) = \hat{r}^{p^{(2)}}(x,C)$. Since $p_n^{(1)}(x) \rightarrow p^{(1)}(x)$ there is an $n_0 \in \mathbb{N}$ such that $p_n^{(1)}(x) \in (0,1)$ for all $n \ge n_0$. Hence, by Proposition~\ref{thm:characterization_optimal_strategies}, we have 
		$\hat{\Pi}^{p_n^{(2)}}\hat{V}^{p_n^{(2)}}(x,C) = \hat{r}^{p_n^{(2)}}(x,C)$ for all $n \ge n_0$, and using \eqref{eq:usc_prelim}, we obtain the desired statement.
		
		(c): Analogous to  (a).
	\end{proof} 
	
	\begin{example}\label{ex:no-NE}
		Let us illustrate that our integrability condition (Assumption~\ref{assum:bounded-functions}) is necessary for the existence of Nash equilibria in that it does not suffice to require \eqref{eq:weak_integrability_condition}. Namely, let us consider a deterministic zero-sum game and set $X_n=n$, $f(n)=g(n)=2^n$, $h(n)=0$ and $\alpha=1/2$. With these choices we recover the example in \cite{Shmaya_Deterministic_Epsilon_Existence}, for which no Nash equilibrium exists (although an $\epsilon$-equilibrium, for every $\epsilon>0$, does exist). We immediately observe that $\mathbb{E}[\sup_{n \in \mathbb{N}_0} \alpha^n f(X_n)]< \infty$. However, $\mathbb{E}[\sup_{n \in \mathbb{N}_0} f(\tilde{X}_n)] = \infty$. This shows how our integrability condition is essential. \defEnd
	\end{example}

	\begin{remark}
		It is well-known that we can interpret any general stochastic process $(X_n)_{n\in\mathbb N_0}$ on a countable space $E$ as a Markov chain on the space of all sample paths $\tilde E:=\{(x_k)_{
			k\leq n}:\,x_k\in E,\,n\in\mathbb N\}$. In this case, our notion of Markovian randomized stopping times coincides with the notion of general behavior stopping times, see \cite[Section 3.2]{solan2012random}. In practical applications, as always with the use of general randomized stopping times, the problem arises that the stopping rules are path-dependent and thus difficult to handle.   Nevertheless, Theorem~\ref{thm:Existence} is applicable and we obtain that in this case a Nash equilibrium (and not only an $\epsilon$-equilibrium) always exists. \defEnd
	\end{remark}

	\subsection{Symmetric Equilibrium Existence for Symmetric Games}\label{sec:countable-symNE}

	For symmetric games we immediately find that the one-player best response mappings $\text{BR}^{(i)},i=1,2$ are identical. 
	Using this observation and arguments analogous to those in the proof of Theorem~\ref{thm:Existence} one can  prove that $\text{BR}^{(1)}$ has a fixed point $p^{(1)}$, which in turn is such that $\left(p^{(1)},p^{(2)}\right)$ with $p^{(1)}=p^{(2)}$ is a global Markovian randomized equilibrium. Indeed we have the following result: 
	
	\begin{theorem}\label{thm:countable-symNE} Let the state space $E$ be countable and consider a symmetric game (cf. \eqref{eq:symmetric-game-def}). 
		Then, a global Markovian randomized equilibrium 
		$\left(p^{(1)},p^{(2)}\right)$ with $p^{(1)}=p^{(2)}$ exists.
	\end{theorem}

	\bibliographystyle{plain}
	\bibliography{MarkovStrategiesDiscreteTimeLiteratureK}
	
\end{document}